\documentclass[10pt
]{article}

\usepackage{amssymb}
\usepackage[dvipdfmx]{graphicx}
\usepackage[usenames]{color}
\usepackage{caption}
\usepackage[dvips]{psfrag}
\usepackage{latexsym}
\usepackage{amsmath}
\usepackage{amsthm}
\usepackage{amsfonts,ascmac,setspace}
\usepackage{amssymb}
\usepackage{tabularx}
\usepackage{algorithm,algorithmic,lscape}
\usepackage{bm}
\usepackage[usenames]{color}
\usepackage{url}

\usepackage[normalem]{ulem}
\usepackage{authblk}

\theoremstyle{plain}
\newtheorem{theorem}{Theorem}[section]
\newtheorem{lemma}[theorem]{Lemma}

\theoremstyle{definition}
\newtheorem{definition}{Definition}
\newtheorem{remark}[theorem]{Remark}

\newtheorem{assumption}[theorem]{Assumption}

\numberwithin{equation}{section}
\numberwithin{theorem}{section}
\numberwithin{table}{section}
\numberwithin{figure}{section}

\title{On the superiority of PGMs to PDCAs \\
in nonsmooth nonconvex sparse regression}
\author{Shummin Nakayama\footnote{E-mail:~\texttt{shummin@kc.chuo-u.ac.jp};\quad ORICD: \texttt{https://orcid.org/0000-0001-7780-8348}}}
\affil[1]{Department of Data Science for Business Innovation, Chuo University
\footnote{1-13-27 Kasuga, Bunkyo-ku, Tokyo 112-8551, JAPAN}}
\author[1]{Jun-ya Gotoh\footnote{E-mail:~\texttt{jgoto@indsys.chuo-u.ac.jp};\quad ORICD: \texttt{https://orcid.org/0000-0002-0097-7298}}}

\begin{document}

\maketitle

\begin{abstract}
This paper conducts a comparative study of proximal gradient methods (PGMs) and proximal DC algorithms (PDCAs) for sparse regression problems which can be cast as Difference-of-two-Convex-functions (DC) optimization problems. It has been shown that for DC optimization problems, both General Iterative Shrinkage and Thresholding algorithm (GIST), a modified version of PGM, and PDCA converge to critical points.  Recently some enhanced versions of PDCAs are shown to converge to d-stationary points, which are stronger necessary condition for local optimality than critical points. 
In this paper we claim that without any modification, PGMs converge to a d-stationary point  not only to DC problems but also to more general nonsmooth  nonconvex problems under some technical assumptions. While the convergence to d-stationary points is known for the case where the step size is small enough, the finding of this paper is valid  also for extended versions such as GIST and its alternating optimization version, which is to be developed in this paper. Numerical results show that among several algorithms in the two categories, modified versions of PGM perform best among those not only in solution quality but also in computation time. 

\noindent
{{\bf Keywords:} Proximal gradient method \and DC algorithms \and  Proximal alternating linearized minimization \and D-stationary points \and Critical points}
\end{abstract}

\section{Introduction}\label{intro}
In regression analysis the variable selection and the outlier detection are crucial to achieve good out-of-sample performance. 
However, popular criteria for those tasks are computationally intractable in that most of them are cast as combinatorial or nonconvex optimization problems. 
Therefore, global optimization methods are applicable only for small-sized instances, and it is common to employ heuristics or approximation approaches. 
 For example, LASSO is the most popular approach to variable selection 
because of the tractability brought by the $\ell_1$-regularizer, but it does not necessarily select ``true" set of variables. 

To fix the limitation of LASSO, there have been proposed several nonconvex formulations such as $\ell_q$-approximation ($0<q<1$), SCAD~\cite{fan2001variable}, MCP~\cite{zhang2010nearly}, Capped-$\ell_1$~\cite{zhang2010analysis}, Log-Sum-Penalty~\cite{candes2008enhancing}, which all employ continuous but nonsmooth nonconvex regularizers to {\it approximate} the so-called $\ell_0$-norm. 

In this paper we consider another 
continuous 
formulation studied by, for example, \cite{ahn2020difference,gotoh2018dc,lu2018sparse}.  
In contrast to the preceding 
`approximation' 
approaches mentioned above, the approach we focus on 
%
has an exact penalty representation. 
Indeed, under some conditions it is shown to have exactly the same optimal solutions as those to the $\ell_0$-constrained counterpart. 

To solve the structured optimization problem, two groups of algorithms have been proposed: (a) those based on Proximal Gradient Method (PGM) and (b) those based on proximal DC (Difference of two Convex functions) Algorithm (PDCA). 
In the literature, each of the algorithms has been compared in part and it is not clear which algorithm performs best. 

The purpose of this paper is to compare several representative implementations from the two groups and to see which is the best for the specially structured task. 
Specifically, findings and contributions of this paper are summarized as follows.
\begin{itemize}
\item 
First of all,
we prove that for the structured 
nonsmooth nonconvex optimization 
 problems, any accumulation point generated by PGM is a d-stationary point, which is a stronger necessary condition for local optimality than the convergence to a critical point, to which PDCA-based algorithms have only been shown to converge except {for} a few enhanced versions. 
%
The convergence is well known for the convex case {and a nonconvex case where the algorithm falls into the Successive Upper-bound Minimization \cite{Razaviyayn2013unified}, but it is not applicable to a more general framework 
where we can incorporate, for example, a nonmonotone line search and/or larger step size}. For example, \cite{gong2013general} deals with a nonconvex case, but it only shows the convergence to a critical point (see Table \ref{tbl:convergence_summary}). 
{To our best knowledge, the convergence of such general PGMs to d-stationary points of a general nonsmooth nonconvex optimization has not yet been documented.}
\item
We extend  the exact penalty representation of \cite{gotoh2018dc,lu2018sparse} to a sparse robust regression, namely, regression analysis seeking a simultaneous pursuit of variable selection and outlier detection.
To solve the problem, we develop GPALM, a nonmonotone extension of Proximal Alternating Linearized Minimization (PALM)~\cite{bolte2014proximal}, and show its convergence to d-stationarity.
\item 
From numerical comparisons over a couple of sparse regression problems, we find that PDCAs perform 
{well} only when the penalty parameter is small whereas they result in worse solutions than PGMs do for large penalty, which is not preferable for the use in the sparse regression problems. 
More specifically, 
General Iterative Shrinkage and Thresholding algorithm (GIST) \cite{gong2013general}, 
which is a modified version of PGM, 
perform better than any other PDCA approaches and plain PGM, not only in computation time, but also in solution quality. 
While modified versions of PDCAs, which are guaranteed to converge to d-stationary points, performed better than plain PDCA in solution quality, it takes longer 
to attain the d-stationarity because of a combinatorial manipulation required at each iteration.
\end{itemize}

\begin{table}[h]
\begin{center}
\caption{Summary of convergence points for the DC problems}
\label{tbl:convergence_summary}
\begin{tabular}{@{~}l@{~}||@{~}c@{~}|@{~}c@{~}|@{~}c@{~}}
\hline
methods & critical point 
&l-stationary point & d-stationary point \\
\hline
PGM    &  $\longleftarrow$     & Attouch et al.~\cite{attouch2013convergence} & this paper \\
PALM   &  $\longleftarrow$     & Bolte et al.~\cite{bolte2014proximal}        & this paper \\
APG    &  $\longleftarrow$     & Lo and Lin~\cite{li2015accelerated}          & this paper\\
GIST & Gong et al.~\cite{gong2013general}& Lu and Li~\cite{lu2018sparse}      & this paper \\
GPALM  &  $\longleftarrow$     & $\longleftarrow$                             & this paper \\
PDCA   & Gotoh et al.~\cite{gotoh2018dc}    & (unknown) & (unknown) \\
PDCAe  & Wen et al.~\cite{wen2018proximal}& (unknown)                     & (unknown) \\
EPDCA  &  $\longleftarrow$     &  $\longleftarrow$                           & Lu et al.~\cite{lu2019enhanced} \\
NEPDCA &  $\longleftarrow$     &  $\longleftarrow$                      & Lu and Zhou~\cite{lu2019nonmonotone} \\
\hline
\end{tabular}
\begin{flushleft}
\scriptsize The d-stationarity implies the l-stationarity, whereas being a critical point is weaker than the other two. See Section \ref{sec:existing_algorithms} for the definition of each solution concept and some illustrative examples. The arrow ``$\longleftarrow$'' implies that the validity of the corresponding convergence is implied by the stronger convergence on its right-hand side.
\end{flushleft}
\end{center}
\end{table}

The structure of this paper is as follows. The next section describes the formulation to be analyzed and a motivation of the analysis in the context of the sparse regression. Section \ref{sec:existing_algorithms} summarizes two categories of existing algorithms and associates those with known convergence results. Section \ref{sec:d-stationarity} is devoted to {showing} the convergence of PGMs to d-stationary points. An extended application to sparse robust regression is discussed in Section \ref{sec:SpaRobReg}. Section \ref{sec:numerics} shows numerical results, comparing the algorithms discussed in this paper. Section \ref{sec:conclusion} concludes the paper. Some proofs are given in Appendix.

\section{Formulation and motivation}
In this paper, we consider several solution methods to solve optimization problems of the form:
\begin{align}
\underset{x}{\mbox{minimize}}\quad F(x) := f(x)+g(x),
\label{eq:opt_prob}
\end{align}
where $f:\mathbb{R}^p\to\mathbb{R}$ is $L$-smooth
  (possibly, nonconvex) and $g:\mathbb{R}^p\to\mathbb{R}\cup\{\infty\}$ is proper lower semi-continuous (\emph{lsc}).
  Here, we say $f$ is $L$-smooth if there exists $L>0$ such that for any $x,y$, it is valid 
  \[\|\nabla f(x)-\nabla f(y)\|\leq L\|x-y\|,\]
where $\|\cdot\|$ denotes the Euclidean norm (or $\ell_2$ norm). 
We assume that $F$ is bounded below. 
In this paper we focus on the case where $g$ is nonsmooth and nonconvex. 

While theoretical results shown in Section \ref{sec:d-stationarity} are valid in a general setting just as stated, 
our focus is on how those algorithms perform in structured sparse optimization problems. 
%
Specifically we further assume that $g$ is DC, namely, it can be represented as the Difference of two Convex functions, i.e., 
$g=g_1-g_2$ with $g_1,g_2$ being both convex, and \eqref{eq:opt_prob} then becomes
\begin{align}
\underset{x}{\mbox{minimize}}\quad F(x) := f(x)+g_1(x)-g_2(x).
\label{eq:opt_prob_dc}
\end{align}
We next explain the details of the problems of interest.
\subsection{Continuous exact penalty reformulation for sparse regression}\label{sec:cxp}
Let $(a_1,b_1),...,(a_N,b_N)\in\mathbb{R}^{p+1}$ be a set of $N$ sampled pairs of input (or feature) vector $\tilde{a}\in\mathbb{R}^p$ and output $\tilde{b}\in\mathbb{R}$, and 
consider to estimate 
a linear 
relation between the input and output:
$\tilde{b}\leftarrow 
\sum_{j=1}^{p}x_j\tilde{a}_{j}
$, 
where 
$x_j$ is the coefficient of the $j$-th input $\tilde{a}_{j},(j=1,...,p)$. 
 (For simplicity, 
 we omit to describe the intercept.) 
A certain empirical risk 
function is usually minimized to estimate the coefficients. 
 For example, minimizing the sum of squared residuals, 
\begin{align*}
f(x)=\sum_{i=1}^{N} (b_i-\sum_{j=1}^{p}a_{ij}x_j)^2,
\end{align*}
is dominantly used for the ordinary regression
, while minimizing the negative likelihood, defined by 
\begin{align*}
f(x)=\sum_{i=1}^{N}\ln\big(1+\exp(-b_i\sum_{j=1}^{p}a_{ij}x_j)\big), 
\end{align*}
of logistic distribution is popular for the 
binary classification 
where $
b_i\in\{\pm1\}$. 
Whatever criterion is used, 
minimizing only an empirical risk function may result in an overfit to the given data set. 
To avoid 
an overfit, 
it is reasonable to consider to select inputs by limiting the number of used coefficients via the $\ell_0$-constrained 
estimation:
\begin{align}
&\underset{
x}{\mbox{minimize}} \quad f(x)
\label{eq:ssr}\\
&\mbox{subject to}                      \quad \|x\|_0\leq K,
\label{eq:l0_constraint}
\end{align}
where $\|x\|_0$ denotes the $\ell_0$ (pseudo-)norm (i.e., the number of non-zero components) of a vector $x$, and $K\in\{0,1,...,N\}$ is a user-defined parameter. 
If $0<K<N$, the $\ell_0$-constraint actually works for variable selection. 
It is widely recognized that the global optimization of \eqref{eq:ssr}--\eqref{eq:l0_constraint} is hard 
because of the combinatorial feature or the discontinuity of the constraint \eqref{eq:l0_constraint}. 
In this paper we tackle 
this problem by using a continuous exact (i.e., not approximate) reformulation studied by, for example, \cite{gotoh2018dc}. 

Let $x_{(i)}$ denote the $i$-th largest element of a vector $x\in\mathbb{R}^{p}$ in absolute values, i.e., $|x_{(1)}|\geq |x_{(2)}|\geq \cdots \geq |x_{(p)}|$. 
Let us denote the sum of the absolute values of the smallest ${p}-K$ components by
\begin{align*}
T_K(x):=
|x_{(K+1)}|+\cdots+|x_{(p)}|.
\end{align*}
It is easy to see that the $\ell_0$-constraint \eqref{eq:l0_constraint} can be replaced with the equality constraint given as $T_K(x)=0$. 
\if
Let $e$ be a closed convex function on $\mathbb{R}$ such that $e(0)=0$ and $e(x)>0$ when $x\neq 0$. 
Among such are $e(x)=x^2$ and $e(x)=|x|$. 
With any $e$ satisfying those properties, 
it is easy to show that \eqref{eq:l0_constraint} can be replaced with the equality constraint:
\begin{align*}
S_K(x):=
\sum\limits_{i=1}^{K}e(x_{(N+1-i)})=0.
\end{align*}
The 
function $S_K(x)$ 
represents 
the sum of the smallest $K$ components of $x$, and 
is generally nonconvex but continuous, unlike $\|x\|_0$. 
In addition, $S_K(x)$ can be represented by a difference of two convex functions. 
 For example, with $e(x)=x^2$, we have $S_K(x)=x_{(N-K+1)}^2+\cdots+x_{(N)}^2$. 
In this paper we focus on the case where $e(x)=|x|$, with which let us define 
\begin{align*}
S_K(x)=T_K(x):=|x_{(N-K+1)}|+\cdots+|x_{(N)}|=\|x\|_1-|\!|\!|x|\!|\!|_{K},
\end{align*}
where $|\!|\!|x|\!|\!|_k$ denotes the sum of the largest $k$ components of $x\in\mathbb{R}^p$ and is called the largest-$k$ norm of $x$, i.e., $|\!|\!|x|\!|\!|_k:=|x_{(1)}|+|x_{(2)}|+\cdots+|x_{(k)}|$. 
\fi
 Here, by moving the constraint to the objective, let us consider 
a penalized version~{\cite{gotoh2018dc}}:
\begin{align}
&\underset{
x}{\mbox{minimize}} \quad 
f(x)+\lambda T_K(x),
\label{eq:penalty_version}
\end{align}
where $\lambda\geq 0$ is a user-defined parameter. 
Since $T_K(x)\geq 0$ for all $x\in\mathbb{R}^p$, by definition, and the second term of the objective function of \eqref{eq:penalty_version} plays a penalty term; 
$T_K(\hat{x})>0$ implies that a point $\hat{x}$ violates the $\ell_0$-constraint, i.e., $\|\hat{x}\|_0>K$. 
What is nice about \eqref{eq:penalty_version} is that it has the exact penalty property. 
Namely, it is shown that for a large $\lambda$, any optimal solution to \eqref{eq:penalty_version} is also optimal to the $\ell_0$-constrained problem \eqref{eq:ssr}--\eqref{eq:l0_constraint}, and vice versa. 
Note that unlike the approaches listed in Introduction, this replacement is not approximation but an equivalent representation to \eqref{eq:l0_constraint}, while maintaining the continuity of the involved function. 
On the other hand, since $T_K$ is nonsmooth and nonconvex similarly to the other continuous approximations, we need to pick up a good solution method. 


Note that \eqref{eq:penalty_version} can be represented as a DC optimization problem of the form \eqref{eq:opt_prob_dc} since $T_K(x)$ can be represented, for example, by
\begin{align}
T_K(x)
=\|x\|_1-|\!|\!|x|\!|\!|_{K},
\label{eq:l1-largeK}
\end{align}
where $|\!|\!|x|\!|\!|_k$ denotes the sum of the largest $k$ components of $x\in\mathbb{R}^p$ and is called the largest-$k$ norm of $x$, i.e., $|\!|\!|x|\!|\!|_k:=|x_{(1)}|+|x_{(2)}|+\cdots+|x_{(k)}|$. 
\section{Existing algorithms}\label{sec:existing_algorithms}
This section describes two groups of algorithms to {approach} 
the nonsmooth nonconvex problems, \eqref{eq:opt_prob} and \eqref{eq:opt_prob_dc}: PGMs (Section \ref{sec:pgm_gist}) and PDCAs (Section \ref{sec:dcas}).
\subsection{PGM and GIST}\label{sec:pgm_gist}
To solve \eqref{eq:opt_prob}, 
PGM updates the incumbent 
$x_{t}$ by the formula:
\begin{align}
x_{t+1}\in {\rm Prox}_{g/\eta_t}\Big(x_t -\frac{1}{\eta_t}\nabla f(x_t)\Big)
\label{eq:PGM}
\end{align}
at each iteration, where {$\eta_0,\eta_1,...,$ are a sequence of positive parameters and}
\begin{align*}
{\rm Prox}_g(y) &:= \underset{x\in\mathbb{R}^p}{\textrm{argmin}}\Big\{g(x)+\frac{1}{2}\|x-y\|^2\Big\}
\end{align*}
is the proximal mapping of $y$ with respect to $g$. The prototype of PGM is described in Algorithm \ref{alg:PGM}. 
If the problem \eqref{eq:opt_prob} is a smooth convex optimization, the condition $\|\nabla F(x_t)\|<\varepsilon$ for a small $\varepsilon>0$ is reasonable for the stopping criterion. However, if it is a nonsmooth or nonconvex problem, the criterion depends on cases. 
As for our problem, it would be reasonable to employ the condition $\|x_t-x_{t-1}\|<\varepsilon$ for a small $\varepsilon>0$, as will be explained.

According to Lu and Li \cite{lu2018sparse}, the proximal {mapping of $x\in\mathbb{R}^p$ 
 with} respect to ${\lambda}T_K$ is given by
\[
\left({\rm Prox}_{{\lambda} T_K}(x)\right)_j
=\left\{
\begin{array}{l@{\quad}l}
x_j,& j\in\mathcal{J},\\
\mathrm{soft}_{\lambda}(x_j),& j\not\in\mathcal{J},\\
\end{array}
\right.
\quad j=1,...,p,
\]
where $\mathcal{J}$ denotes the index set of largest $K$ components {of $x$}, and 
\begin{equation}
\mathrm{soft}_{\lambda}(\xi):=\left\{
\begin{array}{l@{\quad}l}
\xi+\lambda,& \xi\leq -\lambda, \\
0,& -\lambda\leq \xi\leq \lambda,\\
\xi-\lambda,& \xi\geq \lambda,\\
\end{array}
\right.
\label{eq:softthresholding}
\end{equation}
is the soft-thresholding operator {for a component $\xi\in\mathbb{R}$}. 
The operation ${\rm Prox}_{{\lambda}T_K}(x_t)$ can thus be computed efficiently, and it is reasonable to apply PGM to solve \eqref{eq:penalty_version}. 

\begin{algorithm}[tb]
   \caption{Proximal Gradient Method (
PGM)}
   \label{alg:PGM}
\begin{algorithmic}
\STATE {\bf Input:} Problem \eqref{eq:opt_prob}
\STATE {\bf Initialize:} $x_0\in\mathrm{dom}F$, {$\eta_0>0$,} $t\leftarrow0$
\REPEAT \STATE 
{For $\eta_t>0$, update} the incumbent via \eqref{eq:PGM}, and set $t\leftarrow t+1$
\UNTIL{some stopping criterion is satisfied}
\end{algorithmic}
\end{algorithm}
It is known 
 that 
if $f$ is $L$-smooth, 
\if 
$\eta_t>\zeta L$ for some $\zeta>1$, 
any sequence, $\{x_t:t\geq 0\}$, generated by PGM satisfies
\[
F(x_{t+1})\leq F(x_{t})-\frac{(\zeta-1)L}{2}\|x_{t+1}-x_{t}\|^2,
\]
the parameter satisfies { $L<\underline{\eta}<\eta_t<\overline{\eta}$,  where $\underline{\eta}$, $\overline{\eta}$ are constants} 
\fi
any sequence, $\{x_t:t\geq 0\}$, generated by PGM satisfies
\begin{equation}
F(x_{t+1}) \leq F(x_t) - \frac{(1-\frac{L}{\eta_t})\eta_t}{2}\|x_{t+1}-x_{t}\|^2.
\label{eq:suffi_dec}
\end{equation}
When the sequence $\{\eta_t:t\geq 0\}$ satisfies 
\begin{equation}\label{eq:eta-PGM}
L<\hat{L}<\eta_t<\overline{L},\quad t=0,1,...,
\end{equation}
for constants $\hat{L}$ and $\overline{L}$, the term $(1-\frac{L}{\eta_t})\eta_t$ is positive and finite, and 
the sequence $\{F(x_t):t\geq 0\}$ is monotonically decreasing. 
Furthermore, if $F$ is bounded below and the generated sequence $\{x_t\}$ is bounded, then it follows from  \eqref{eq:suffi_dec} that
\begin{align}\label{eq:lim_pgm}
\lim_{t\to\infty}\|x_{t+1}-x_t\|=0.
\end{align}
Consequently, $\{x_t\}$ clusters at a 
point satisfying 
\begin{equation}
0\in 
\partial F(x^\ast) =
\nabla f(x^\ast) + \partial g(x^\ast),
\label{def:stationary_point}
\end{equation}
(see, e.g., Section 5 of \cite{attouch2013convergence}, for the details). 
Here we should recall the definition of subdifferential $\partial F(x)$ of 
a nonconvex function $F$. 
\begin{definition}
 For a proper lower semi-continuous (lsc) function $F:\mathbb{R}^p\to(-\infty,\infty]$, the \emph{limiting subdifferential} {of $F$ at $\bar{x}\in\mathrm{dom}\,F$} is defined as
\begin{align*}
\partial F({\bar{x}})&:=\left\{ \xi \in {\mathbb R}^p \left|
\begin{array}{l}
\exists (x_t,F(x_t)) \to (\bar{x},F(\bar{x})), \\
\quad \xi_t \in \hat\partial F(x_t),~\xi_t \to \xi 
\end{array}
\right.\right\},
\end{align*}
where 
\begin{align*}
\hat\partial F(x)&\!:=\!\left\{ \xi \in {\mathbb R}^p \left| \liminf_{
y \to x}\frac{F(y)\!-\!F(x)\!-\!\langle\xi,y-x\rangle}{\|y-x\|}\! \geq \!0 \right.\right\}
\end{align*}
is called \emph{regular subdifferential}. 
\end{definition}
By definition we have that $\hat\partial F(x) \subset \partial F(x)$. 
Note that when $g$ is a proper lsc convex function, 
both 
$\partial F(x)$ and $\hat\partial F(x)$ 
coincide with the ordinary subdifferential of the convex function, and then we do not have to pay attention to the difference of the two notions of subdifferentials. On the other hand, however, when $g$ is nonconvex and nonsmooth, the difference can be significant. 
For example, $F(x)=\frac{1}{2}x^2-|x|$ is indifferentiable and nonconvex around $x=0$, so that $\hat{\partial}F(0)=\emptyset$. On the other hand, $\partial F(0)=\{-1,1\}$. 

With these 
subdifferentials, we 
can define the two notions of stationarity. 
\begin{definition}
We call $x^\ast$ a \emph{l(imiting)-stationary point} of  \eqref{eq:opt_prob} if $0\in \partial F(x^\ast)$
. On the other hand, we call $x^\ast$ a \emph{(regular) stationary point} of \eqref{eq:opt_prob} if $0\in \hat{\partial} F(x^\ast)$. 
\end{definition}
Obviously, any regular stationary point is l-stationary (Figure \ref{fig:relations_among_stationary_points}). 
It is noteworthy that all the existing works on PGM for nonconvex problems only show the convergence to l-stationary points (see, e.g., Table \ref{tbl:convergence_summary}). 

As a more intuitive necessary condition of local optimality, we introduce the directional stationarity. 
Let us denote the \emph{directional derivative} of $F$ at $x$ with respect to a direction $d\in\mathbb{R}^p$ by
\begin{align*}
F'(x;d) &:= \lim_{\tau\to+0}\frac{F(x+\tau d)-F(x)}{\tau}.
\end{align*}
Then, $F$ is said to be directionally differentiable at $x$ if $F'(x;d)$ exists for any $d$, and $F$ is simply said to be directionally differentiable if $F'(x;d)$ exists for any $d$ and $x\in\mathrm{dom}F$. 
(Here, $F'(x;d)$ may take $\infty$.) 
 
\begin{definition}
We call $x^\ast$ a \emph{d(irectional)-stationary point} of \eqref{eq:opt_prob} if 
$F$ is directionally differentiable at $x^\ast$, and 
\begin{equation}
F'(x^\ast;d) \geq 0\quad\mbox{for all }d\in\mathbb{R}^p.
\label{eq:def:directional-derivative}
\end{equation}
\end{definition}

It is known 
that 
when $F$ is directionally differentiable 
 and locally Lipschitz continuous 
 at $x^\ast$, 
the regular stationarity at $x^\ast$ is equivalent to 
the d-stationarity at $x^\ast$. 
Accordingly, for any directionally differentiable locally Lipschitz continuous function, the d-stationarity is a stronger necessary condition for local optimality of \eqref{eq:opt_prob} than the l-stationarity (see, e.g., \cite{cui2018composite,Li2020understanding} for the details).

{
It is shown in 
\cite{Razaviyayn2013unified} that a class of Successive Upper-bound Minimization (SUM) methods converge to d-stationary points for the nonsmooth nonconvex problems \eqref{eq:opt_prob} when $F$ is directionally differentiable. 
When $\{\eta_t\}$ satisfies \eqref{eq:eta-PGM}, the plain PGM falls into the class and its d-stationarity is shown in the SUM framework. 
On the other hand, PGM with a (nonmonotone) line search \eqref{eq:nmt} 
 is not an SUM method, and hence, we 
%
next 
consider the PGM with a nonmonotone line search aka GIST. 
}

In addition to the proximal operation \eqref{eq:PGM} at each iteration of PGM, GIST (Algorithm \ref{alg:GIST}) employs the Barzilai-Borwein (BB) rule \cite{barzilai1988two} to set the (initial) step size 
${
\hat{\eta}_t
}$ and a nonmonotone line search. 
More specifically, at each iteration GIST initializes the step size by the formula:
\begin{align}
\hat{\eta}_t
=
\min\left\{\overline{\eta}, \max\left\{\underline{\eta},\frac{\langle x_t-x_{t-1},\nabla f(x_t)-\nabla f(x_{t-1}) \rangle}{\|x_t-x_{t-1}\|^2}\right\}\right\},~0<\underline{\eta}<\overline{\eta},
\label{BB}
\end{align}
so that it reflects the curvature of $f$, and adjusts the step size 
by 
backtracking 
$\eta_t=\rho^l
{
\hat{\eta}_t
},l=0,1,...,$ with $\rho>1$, until satisfying 
\begin{align}
F(x_{t+1})\leq \max\{F(x_{t-r+1}),...,F(x_t)\}
 -\frac{\sigma\eta_t}{2}\|x_{t+1}-x_t\|^2,
\label{eq:nmt}
\end{align}
where $\sigma\in(0,1)$ and an integer 
$r\geq 1$
 to approximately ensure the decrease of the optimal values. 

{The condition \eqref{eq:nmt} is modified from \eqref{eq:suffi_dec} in two ways. 
 First, the monotonicity of the decrease of $F(x_t)$ is relaxed when $r\geq 2$ in \eqref{eq:nmt}, while setting $r=1$ imposes the monotonic decrease. 
Second, unlike the (plain) PGM with \eqref{eq:eta-PGM}, GIST allows us to take a longer step size by setting $\eta_t<L$, as long as \eqref{eq:nmt} is fulfilled. 

While the condition $\eta_t<L$ might not satisfy \eqref{eq:nmt}, increasing $\eta_t$ with the backtracking ensures its fulfillment. 
To see this, observe that \eqref{eq:nmt} is always satisfied with $\eta_t$ such that $L<\eta_t$. Note also that $\eta_t$ is increasing and $L<\eta_t<\rho L$ is attained in a finite number of multiplications for the backtracking. Therefore, the condition \eqref{eq:nmt} will be fulfilled eventually. 

Note that this also shows that the sequence $\eta_t$ is bounded, namely, $\hat{\eta}_t\leq \eta_t<\rho L$, which, coupled with a constant $\sigma$, ensures the sufficient decrease as in \eqref{eq:suffi_dec}. (Comparing with \eqref{eq:suffi_dec}, we see that the constant $\sigma\in(0,1)$ corresponds to $1-L/\eta_t$, which is in the interval $(0,1)$ under \eqref{eq:eta-PGM}.)} 
\begin{algorithm}[tb]
   \caption{General Iterative Shrinkage Thresholding (GIST) \cite{gong2013general}}
   \label{alg:GIST}
\begin{algorithmic}
\STATE {\bf Input:} Problem \eqref{eq:opt_prob}; $\rho>1$; $0<\underline{\eta}<\overline{\eta}$
\STATE {\bf Initialize:} $x_0\in\mathrm{dom}F$, $t\leftarrow0$
\REPEAT \STATE Choose ${\hat{\eta}_t}\in[\underline{\eta},\overline{\eta}]$ {and set $\eta_t=\hat{\eta}_t$}.
 \REPEAT
 \STATE $x_{t+1}\in {\rm Prox}_{g/\eta_t}\left(x_t -\frac{1}{\eta_t}\nabla f(x_t)\right)$
 \STATE $\eta_t\leftarrow \rho\eta_t$
 \UNTIL{{line search condition \eqref{eq:nmt} holds.}
 }
\STATE Set $t\leftarrow t+1$.
\UNTIL{some stopping criterion is satisfied}
\end{algorithmic}
\end{algorithm}

Gong et al.~\cite{gong2013general}  shows that 
if the sequence $\{x_t\}$ generated by Algorithm \ref{alg:GIST} is bounded, $F$ is continuous on a compact set containing the sequence, and $F$ is bounded below, then \eqref{eq:lim_pgm} holds. 
Furthermore, 
when $g$ is represented as a difference of two convex functions as in \eqref{eq:opt_prob_dc}, GIST  
clusters at a \emph{critical point}.
\begin{definition}
We call $x^*$ a \emph{critical point} of the DC optimization problem \eqref{eq:opt_prob_dc} if it satisfies
\begin{equation}
0\in \nabla f(x^\ast) + \partial g_1(x^\ast)-\partial g_2(x^\ast).
\label{def:critical_point}
\end{equation}
\end{definition}
\begin{remark}
The term ``critical point" is often used as the synonym for ``stationary point." However, we distinguish these terms in this paper, following the convention of the DC optimization literature.
\end{remark}
\begin{figure}[h]
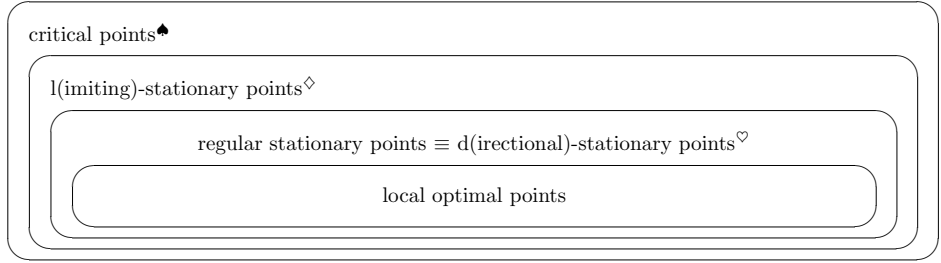

{\centering
\scalebox{0.8}{
\begin{screen}
critical points$^\spadesuit$
\begin{screen}
l(imiting)-stationary points$^\diamondsuit$ 
\begin{screen}\centering
regular stationary points 
 $\equiv$ d(irectional)-stationary points$^\heartsuit$
\begin{screen}\centering 
local optimal points
\end{screen}\vspace{-.5em}
\end{screen}\vspace{-.5em}
\end{screen}\vspace{-.5em}
\end{screen}
}
\caption{Relations of different stationarity notions}
\label{fig:relations_among_stationary_points}
}
\begin{flushleft}
\scriptsize $^\spadesuit$We distinguish ``critical point'' from ``stationary point,'' following the convention of the DC optimization literature; 
$^\diamondsuit$This is often referred to just as `stationary point' for nonsmooth nonconvex optimization; 
$^\heartsuit$`$\equiv$' holds when $F$ is locally Lipschitz continuous near the point and directionally differentiable (see, e.g., {\cite{cui2018composite,Li2020understanding}}).
\end{flushleft}
\end{figure}
Note that the convergence of GIST is not proved in 
\cite{gong2013general} for general nonconvex case \eqref{eq:opt_prob}, but for the DC case \eqref{eq:opt_prob_dc}. 

To demonstrate the difference between the critical point \eqref{def:critical_point} and the l-stationary point \eqref{def:stationary_point}, let us consider $F(x)=f(x)+g_1(x)-g_2(x)$, where $f(x)=\frac{1}{2}(x-2)^2$, $g_1(x)=|x|$, and $g_2(x)=\max\{0,-x\}$. 
We give a plot of $F$ and its slopes in Figure \ref{fig:crit}.
It is easy to see that $x=1$ is the unique minimizer. Observing that $0 \in [-3
,0] =f'(0) + \partial g_1(0) - \partial g_2(0)$, we see from \eqref{def:critical_point} that $x=0$ is a critical point of $F$. However, it is not l-stationary, 
and, actually, $F(x)$ is decreasing around $x=0$ 
since $F(x)=\frac{1}{2}(x-2)^2 + \max\{0,x\}$ and $0 \not\in \partial F(0)=[-2
,-1]$. 
This indicates that when the DC decomposition, $g=g_1-g_2$, is given by two ``non-smooth" functions, $g_1$ and $g_2$, its critical point \eqref{fig:relations_among_stationary_points} might be quite different from any ``stationarity'' because of the difference, $\partial g_1-\partial g_2$, {excessively} relaxes the subgradient $\partial g=\partial (g_1-g_2)$. 
This gap 
plays a significant role, as below, in characterizing the behaviors of the algorithms. 
\begin{figure}[h]
  \begin{center} 
       \includegraphics[
scale=0.5]
{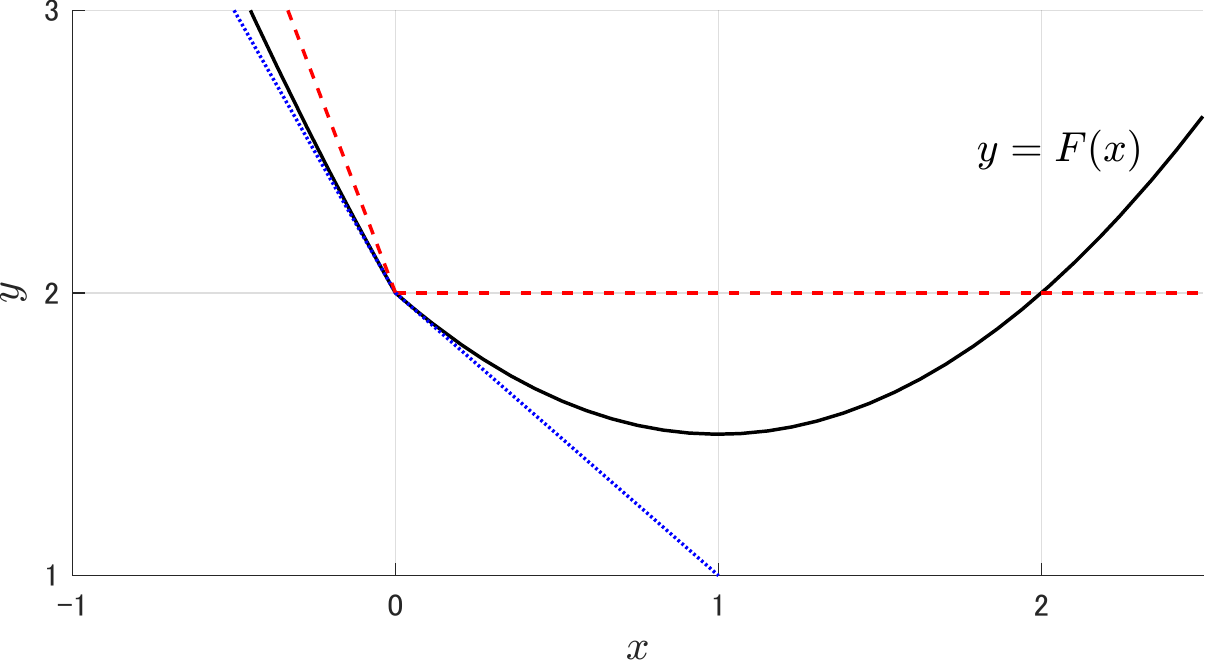}
    \caption{Example of a critical point that is not stationary}
    \label{fig:crit}
  \end{center}
\begin{flushleft}
\scriptsize The point $x=0$ is a critical point of $F=f+g$ under the DC decomposition $g=g_1-g_2$ with $g_1(x)=|x|$ and $g_2(x)=\max\{0,-x\}$, but it is not d-stationary: The (blue) dotted lines indicate the slopes corresponding to the range $[-2,-1]$, the right-hand side of \eqref{def:stationary_point}, whereas the (red) dashed lines indicate the slopes corresponding to $[-3,0]$, the right-hand side of \eqref{def:critical_point} at $x^*=0$. 
\end{flushleft}
\end{figure}

\subsection{PDCA and its extensions}
\label{sec:dcas}
 For the case where $g$ has the DC decomposition as in \eqref{eq:opt_prob_dc}, another method called DCA can be used. 
The basic strategy repeats the two steps: 1) linearization (or computation of a subgradient) of $g_2$; 2) solve the resulting convex subproblem. 
This suggests the update formula of the form:
\begin{equation}
x_{t+1} = {\rm Prox}_{g_1/L}\left(x_t-\frac{1}{L}\left(\nabla f(x_t)-\xi_t\right)\right),
\label{eq:prox_PDCA}
\end{equation}
where $\xi_t\in\partial g_2(x_t)$. 
Especially for \eqref{eq:penalty_version} with the DC decomposition \eqref{eq:l1-largeK}, the formula \eqref{eq:prox_PDCA} results in a soft-thresholding \eqref{eq:softthresholding}. 
In addition, the computation of a subgradient of $g_2(x)=\lambda|\!|\!|x|\!|\!|_K$ can be done efficiently. 
This algorithm is called \emph{PDCA} (Algorithm \ref{alg:PDCA}). 
The PDCA converges to a critical point. 

Several modified algorithms have been recently proposed. 
 For example Wen et al.~\cite{wen2018proximal} proposes \rm{PDCAe} (Algorithm \ref{alg:PDCAe}), 
 adding an extrapolation step for possible acceleration, and shows the convergence to a critical point when $g_1$ is a proper closed convex function, $g_2$ is a continuous convex function, and $F$ is level bounded.
\begin{algorithm}[tb]
   \caption{Proximal DC Algorithm (PDCA)\cite{gotoh2018dc}}
   \label{alg:PDCA}
\begin{algorithmic}
\STATE {\bf Input:} Problem \eqref{eq:opt_prob_dc}
\STATE {\bf Initialize:} $x_0\in\mathrm{dom}F$, $t\leftarrow0$
\REPEAT 
\STATE Compute $\xi_t\in\partial g_2(x_t)$
\STATE Update the incumbent via \eqref{eq:prox_PDCA}, and $t\leftarrow t+1$
\UNTIL{some termination condition holds.}
\end{algorithmic}
\end{algorithm}
\begin{algorithm}[tb]
   \caption{
PDCA with extrapolation (PDCAe) \cite{wen2018proximal}}
   \label{alg:PDCAe}
\begin{algorithmic}
\STATE {\bf Input:} Problem \eqref{eq:opt_prob_dc}, $\beta_t\in[0,1),\sup_t \beta_t<1$
\STATE {\bf Initialize:} $x_0\in\mathrm{dom}g, x_{-1}=x_0$, $t\leftarrow0$
\REPEAT 
\STATE Take any $\xi_t\in\partial g_2(x_t)$.
\STATE Set $y_t=x_t+\beta_t(x_t-x_{t-1})$.
\STATE Update the incumbent via 
\begin{center}$
x_{t+1} = {\rm Prox}_{g_1/L}\left(y_t-\frac{1}{L}\left(\nabla f(y_t)-\xi_t\right)\right),
$\end{center}
and $t\leftarrow t+1$.
\UNTIL{some termination condition holds.}
\end{algorithmic}
\end{algorithm}

Recently 
some modified versions of PDCA have been developed, generating a sequence converging to a d-stationary point \eqref{eq:def:directional-derivative}, which is a stronger condition for local optimality than critical point or l-stationary point. 
Pang et al.~\cite{pang2017computing} introduces an enhanced DCA to ensure the d-stationarity when the convex function $g_2$ is given by the pointwise maximum of finite number of differentiable convex functions. 
Inspired by the algorithm, Lu et al.~\cite{lu2019enhanced} 
develops the Enhanced PDCA (\emph{EPDCA}), which 
retains the convergence to a d-stationary point. 
Lu and Zhou~\cite{lu2019nonmonotone} develops Nonmonotone EPDCA (NEPDCA; Algorithm \ref{alg:NEPDCA}), by further employing the nonmonotone line search and the BB rule as in GIST 
and/or incorporating a randomization of the selection from the active set associated with $g_2$, while retaining the convergence to a d-stationary point. 

To ensure the d-stationarity, 
$g_2$ is further assumed to be given as the pointwise maximum of a finite number of differentiable convex functions, $\gamma_i,i\in\mathcal{I}$, (i.e., $g_2(x)=\max\{\gamma_i(x):i\in\mathcal{I}\}$ for some $\mathcal{I}$). 
Under this modified setting, the d-stationarity 
of \eqref{eq:opt_prob_dc} 
can be characterized as a point satisfying
\[
\forall i\in\mathcal{I},~
0\in\nabla f(x^*)+\partial g_1(x^*)-\nabla \gamma_i(x^*).
\]
Roughly speaking, NEPDCA looks for a point that satisfies stationarity conditions of all the $|\mathcal{I}|$ convex subproblems. 
The key is how to update the (relaxed) set of active functions at $x_t$:
\[
A_\delta(x):=\{i\in\mathcal{I}:\gamma_i(x)\geq g_2(x)-\delta\},
\]
where $\delta\geq 0$ represents a degree of relaxation, so that the cardinality of $A_\delta(x)$ cannot decrease as $\delta$ increases. 
As will be reported, when it is applied to the problem \eqref{eq:penalty_version}, a possible bottleneck of Algorithm \ref{alg:NEPDCA} is in constructing the active set. 
For \eqref{eq:penalty_version} with \eqref{eq:l1-largeK}, $\gamma_i$ is given as 
$
\gamma_i(x) = \langle v,x \rangle
$ 
for some $v$ such that $v_{j} =\mathrm{sign}(x_j)$ if $|x_j|\geq|x_{(K)}|$; ~$0$, otherwise.
Recalling that
\[
|\!|\!|x|\!|\!|_K
 = \max\limits_{v}\big\{ \langle v,x \rangle : \sum\limits_{j=1}^p|v_j|=K,v\in\{0,\pm 1\}^p\big\},
\]
the cardinality of $A_\delta(x)$ can be enormous even for a small $\delta$. 
For example, 
when $\delta=0$ and $|x_{(K-j-1)}|>0=|x_{(K-j)}|=\cdots=|x_{(K)}|=\cdots=|x_{(p)}|$ hold,
$|A_\delta(x)|$ turns out to be $\frac{(p-K+j)!}{j!(p-K)!}$. 
Besides, for a larger $\lambda$ in \eqref{eq:penalty_version}, the soft-thresholding operation returns more zero components, and accordingly 
the largest $K$ components of $x$ can include many zeros, which would lead to the increase in time for computing $A_\delta$. 

Although Lu and Zhou~\cite{lu2019nonmonotone} also developed a random sampling version to improve the efficiency,
at a little cost of deterioration in the quality of the optimal solution, 
we only consider the non-randomized version in this study.
\begin{algorithm}[tb]
   \caption{A deterministic nonmonotone enhanced PDCA with line search \cite{lu2019nonmonotone}}
   \label{alg:NEPDCA}
\begin{algorithmic}
\STATE {\bf Input:} Problem \eqref{eq:opt_prob_dc} where $g_2(x)=\max\{\gamma_i(x):i\in \mathcal{I}\}$; $\delta>0,\rho>1,c\in(0,L/2),0<\underline{\eta}<\overline{\eta}$, $r\in\mathbb{N}$
\STATE {\bf Initialize:} $x_0\in\mathrm{dom}F, t\leftarrow 0$
\REPEAT 
 \STATE 
 Choose ${\eta_{t}}\in[\underline{\eta},\overline{\eta}]$.
  \REPEAT 
   \FOR{$i\in \mathcal{A}_{\delta}(x_t)$}
    \STATE Compute $x_{k,i}(\eta_t)=$
    \begin{equation*}
    {\rm Prox}_{g_1/{\eta_t}}\left(x_t-\frac{1}{\eta_t}\left(\nabla f(x_t)-\nabla\gamma_i(x_t)\right)\right).
    \label{eq:NEPDCA_prox}
\end{equation*}
   \ENDFOR
  \STATE $\hat{i}\in\underset{i\in\mathcal{A}_{\delta}(x_t)}{\textrm{argmin}}\{F(x_{t,i}(\eta_t))+\frac{c}{2}\|x_{t,i}-x_t\|^2\}$
  \IF{$x_{t,\hat{i}}$ satisfies
  \begin{align}
  F(x_{t,\hat{i}}(\eta_t))&\leq\max\Big\{f(x_t)+g_1(x_t)-\gamma_i(x_t),
  \max_{\max\{t-r,0\}\leq j\leq t}F(x_j)\Big\}\notag\\
  &\quad-\frac{c}{2}\|x_{t,\hat{i}}(\eta_t)-x_t\|^2
  -\frac{c}{2}\|x_{t,i}(\eta_t)-x_t\|^2,~\forall i\in\mathcal{A}_{\delta}(x_t)\label{eq:NEPDCA_line}
  \end{align}
  }
  \STATE Set $x_{t+1}\leftarrow x_{t,\hat{i}}(\eta_t)$
  \ENDIF
   \STATE {
   $\eta_t\leftarrow \rho\eta_t$
   }
 \UNTIL{line search condition \eqref{eq:NEPDCA_line} holds.}
\STATE Set $t\leftarrow t+1$
\UNTIL{some termination condition holds.}
\end{algorithmic}
\end{algorithm}

\section{PGM's convergence to d-stationary point}
\label{sec:d-stationarity}
In this subsection we show that PGMs converge to d-stationary points. 
To that end, we make the following assumptions for problem \eqref{eq:opt_prob}.
\begin{assumption}
\label{assump:d-stationary}~
\begin{enumerate}
\item[(i)]
 $f$ is $L$-smooth and $g$ is proper \emph{lsc} and directionally differentiable.
\item[(ii)]
$F$ is bounded below.
\item[(iii)]
$g$ is prox-bounded, i.e., $g+\frac{\eta}{2}\|\cdot\|^2$ is lower bounded for some $\eta>0$.
\end{enumerate}
\end{assumption}
Note that Assumption \ref{assump:d-stationary}(i) implies that $F$ is also directionally differentiable. Also, Assumption \ref{assump:d-stationary}(iii) implies that the proximal mapping is outer semicontinuous \cite[Example 5.23(b)]{rockafellar2009variational}:
\[\limsup_{x \to \bar{x}} {\rm Prox}_{g/\eta}(x) \subseteq {\rm Prox}_{g/\eta}(\bar{x}).
\]
The outer semicontinuity plays an important role in showing the d-stationarity of PGMs. 
Note that in the sparse regression problem \eqref{eq:penalty_version}, Assumption \ref{assump:d-stationary}(iii) is satisfied since $T_K(x)\geq0$; 
the other penalty functions such as SCAD~\cite{fan2001variable}, MCP~\cite{zhang2010nearly}, Capped-$\ell_1$~\cite{zhang2010analysis}, and Log-Sum-Penalty~\cite{candes2008enhancing} also satisfy 
the condition. 

We first give the following lemma.
\begin{lemma}\label{lemma:for_d-stat}
Let
$
\epsilon_t : = \eta_t(x_{t+1} - x_t ) +\nabla f(x_t) -\nabla f(x_{t+1})
$, 
and $x_{t+1}$ be given by \eqref{eq:PGM}. 
Suppose that Assumption \ref{assump:d-stationary}(i) is satisfied. 
Then, we have for any $d\in\mathbb{R}^p$, 
\begin{align}
-
\|\epsilon_t\|
\|d\|\leq F'(x_{t+1};d).
\end{align}
\end{lemma}
See Section \ref{sec:proof_lem:d-stat} for the proof.
By using Lemma \ref{lemma:for_d-stat} and \eqref{eq:lim_pgm}, we have the following main theorem.
\begin{theorem}
\label{thm:d-stat_PGM}
Suppose that Assumption \ref{assump:d-stationary} holds. 
Let the sequence $\{(x_t,\eta_t):t\geq 0\}$ be generated by Algorithm \ref{alg:GIST} (when the termination condition is ignored). 
If 
the generated sequence $\{x_t\}$ is bounded and $F$ is continuous on a compact set containing the sequence, then any accumulation point of $\{x_t\}$ is a d-stationary point of \eqref{eq:opt_prob}. 
\end{theorem}
Note that since neither Lemma \ref{lemma:for_d-stat} nor Theorem \ref{thm:d-stat_PGM} assumes the locally Lipschitz continuity of $F$, Theorem \ref{thm:d-stat_PGM} does not claim the regular stationarity $0\in \hat{\partial}F(x^\ast)$ but the d-stationarity \eqref{eq:def:directional-derivative}. On the other hand, if $F$ is locally Lipschitz continuous, Theorem \ref{thm:d-stat_PGM} ensures the regular stationarity.  
\begin{proof}
Let $x^*$ be any accumulation point of $\{x_t\}$ and let $\{x_{t_i}\}$ be a convergent subsequence with $x_{t_i}\to x^*$. By passing to a further subsequence if necessary, we also assume that $\eta_{t_i}\to\eta^*$ for some $\eta^*\in(\underline{\eta},\overline{\eta})$. 
To prove the theorem, it is enough to show
\begin{equation}\label{eq:x^*in}
x^*\in {\rm Prox}_{g/\eta^*}\left(x^* - \frac{1}{\eta^*}\nabla f(x^*)\right),
\end{equation}
since it follows from Lemma \ref{lemma:for_d-stat} with $x_{t+1}=x_t=x^*$ and $\eta_t=\eta^*$ that $F'(x^*;d)\geq 0$ for all $d\in\mathbb{R}^p$. 

From $\eta_{t_i}\to\eta^*$
and Assumption \ref{assump:d-stationary}(iii), we have
\[\limsup_{i\to \infty} {\rm Prox}_{g/\eta_{t_i}}\left(x_{t_i}- \frac{1}{\eta_{t_i}}\nabla f(x_{t_i})\right) \subseteq {\rm Prox}_{g/\eta^*}\left(x^* - \frac{1}{\eta^*}\nabla f(x^*)\right).\]
This follows from Exercise 7.38 of \cite{rockafellar2009variational} together with the facts that ${x_{t_i}}- \frac{1}{\eta_{t_i}}\nabla f(x_{t_i})$ and 
\[
\limsup_{i\to \infty} {\rm Prox}_{g/{\eta_{t_i}}}\Big(x_{t_i}- \frac{1}{\eta_{t_i}}\nabla f(x_{t_i})\Big)\subseteq {\rm g\mathchar`-}\limsup_{i} {\rm Prox}_{g/{\eta_{t_i}}} \Big( x^*-\frac{1}{\eta^*}\nabla f(x^*) \Big), 
\]
where the latter follows from Proposition 5.33 of \cite{rockafellar2009variational}.
It follows from \eqref{eq:PGM} and \eqref{eq:lim_pgm} that
\[
x^*=\lim_{i\to\infty}{x_{t_i+1}}\in\limsup_{i\to \infty} {\rm Prox}_{g/\eta_{t_i}}\left(x_{t_i}- \frac{1}{\eta_{t_i}}\nabla f(x_{t_i})\right).
\]
Therefore, we have \eqref{eq:x^*in}, which means that $x^*$ is a d-stationary point. 
\end{proof}

\begin{remark}
In addition to GIST, any PGM-based algorithms satisfying \eqref{eq:lim_pgm} converge to d-stationary points. Among such are plain PGM (Algorithm \ref{alg:PGM} with \eqref{eq:eta-PGM}) and  Accelerated Proximal Gradient method (APG) \cite{li2015accelerated}. 
While the convergence of PGM with a small step size \eqref{eq:eta-PGM} can also be proven in the SUM framework as mentioned in Section \ref{sec:pgm_gist}, 
Theorem \ref{thm:d-stat_PGM} applies to all such PGMs. 
(see Table \ref{tbl:convergence_PGM}).
\end{remark}

\begin{table}[h]
\begin{center}
\caption{Summary of PGM's convergence points for nonsmooth nonconvex problems}
\label{tbl:convergence_PGM}
\begin{tabular}{@{~}l@{~}||@{~}c@{~}|@{~}c@{~}}
\hline
PGM with & l-stationary point &d-stationary point    \\
\hline
\hline
line search   & Lu and Li~\cite{lu2018sparse}  &  this paper\\
\hline
condition \eqref{eq:eta-PGM}   & Attouch et al.~\cite{attouch2013convergence} &  this paper and Razaviyayn et al.~\cite{Razaviyayn2013unified}  \\
\hline
\end{tabular}
\end{center}
\end{table}

\section{Sparse robust regression}
\label{sec:SpaRobReg}
In this section we extend the methodology developed in this paper by combining the variables selection and outlier detection. 

In real practice of regression analysis, 
we are often suggested to remove outlying samples 
in estimating the model. 
Least Trimmed Squares method (LTS) is among such methodologies which are called {\it robust regression} and can be formulated with the help of a similar idea to the variable selection described in Section \ref{sec:cxp}. 
 Furthermore, those outlier detection methods can be coupled with the sparse regression. 
Consider the formulation proposed by Liu et al.~\cite{liu2019Arefined}. 
\begin{align}
\underset{
x, z}{\mbox{minimize}} \quad &  f(x,z):=\sum_{i=1}^{N} (b_i-
\sum_{j=1}^{p}a_{ij}x_j-z_i)^2 \label{eq:spa.rob.obj}\\
\mbox{subject to}   \quad & 
\|x\|_0\leq K,\quad \|z\|_0\leq \kappa.
\label{eq:ls_2l0}
\end{align}
While the first $\ell_0$-constraint suppresses the number of selected variables (or coefficients), the second one suppresses the number of outlying samples since $z_i=b_i-\sum_{j=1}^{p}a_{ij}x_j\neq 0$ implies that the sample $i$ is discarded from the least square estimation. 
To approach a solution of \eqref{eq:spa.rob.obj}--\eqref{eq:ls_2l0},  Liu et al.~\cite{liu2019Arefined} consider
\begin{equation}\label{eq:LiuRSR}
\underset{
x,z}{\mbox{minimize}} \quad 
f(x,z)+\lambda' T_K(x)+
\iota_{\Omega}(z),
\end{equation}
\noindent
where $\lambda'\geq 0$ is a 
penalty parameter on the violation $T_K(x)>0$ of the first $\ell_0$-constraint of \eqref{eq:ls_2l0} and $\iota_{\Omega}(z)$  is the indicator function of $\Omega=\{z \in\mathbb{R}^N~|~ \|z\|_0\leq \kappa\}$,
and proposes an alternating optimization approach in which PDCAe is applied to the optimization with respect to $x$ and the projection to $\Omega$ is used for the $z$-update. 
In contrast, 
we solve another penalty form:
\begin{align}
\underset{
x,z}{\mbox{minimize}} \quad 
f(x,z)+\lambda_1T_K(x)+\lambda_2T_\kappa(z),
\label{eq:spa.rob.reg_pen}
\end{align}
where $\lambda_1$ and $\lambda_2$ are nonnegative constants for penalties on the violations $T_K(x)>0$ and $T_\kappa(z)>0$, respectively, of the associated $\ell_0$-constraints \eqref{eq:ls_2l0}.

Here let us show the equivalence between \eqref{eq:spa.rob.reg_pen} and 
\eqref{eq:spa.rob.obj}--\eqref{eq:ls_2l0} in a generalized case where the quadratic objective function \eqref{eq:spa.rob.obj} is replaced with a general function $f$ satisfying the following smoothness condition: 
\begin{assumption}
\label{assump:L-smooth_xz}
There exists $M>0$, for any $x_1,x_2\in\mathbb{R}^p$, $z_1,z_2\in\mathbb{R}^N$,
\begin{align*}
\left\|
\begin{matrix}
\nabla_x f(x_1,z_1) -\nabla_x f(x_2,z_2)\\
 \nabla_z f(x_1,z_1)-\nabla_z f(x_2,z_2)
 \end{matrix}\right\|
\leq M
\left\|
\begin{matrix}
x_1-x_2\\
z_1-z_2
\end{matrix}
\right\|.
\end{align*}
\end{assumption}
It is easy to see that the quadratic function \eqref{eq:spa.rob.obj} satisfies this condition. 
 For such $f$, we can prove the existence of exact penalty parameters, with which \eqref{eq:spa.rob.reg_pen} is equivalent to the minimization of $f(x,z)$ under \eqref{eq:ls_2l0}, as below. 
\begin{theorem}
\label{thm:exact.pen_spa.rob.reg}
Let $f:\mathbb{R}^p\times\mathbb{R}^N\to\mathbb{R}$ be any function satisfying Assumption \ref{assump:L-smooth_xz}, and 
let $S(\lambda_1,\lambda_2)\subset\mathbb{R}^p\times\mathbb{R}^N$ denote the set of optimal solutions to \eqref{eq:spa.rob.reg_pen} for $(\lambda_1,\lambda_2)$. 
Suppose that there exist $\underline{\lambda}_1,\underline{\lambda}_2\geq 0$ such that $S(\underline{\lambda}_1,\underline{\lambda}_2)$ is bounded, and let $C_x$ and $C_z$ be constants such that $\|\underline{x}\|\leq C_x$ and $\|\underline{z}\|\leq C_z$ for any $(\underline{x},\underline{z})\in S(\underline{\lambda}_1,\underline{\lambda}_2)$. 
Then, for $\lambda_1,\lambda_2$, such that 
\begin{align}
\lambda_1&>
\max\Big\{\|\nabla_{x}f(0,0)\|+M(\frac{3}{2}C_x+C_z),\underline{\lambda}_1\Big\},
\label{eq:cond_lambda1}\\
\lambda_2&>
\max\Big\{\|\nabla_{z}f(0,0)\|+M(C_x+\frac{3}{2}C_z),\underline{\lambda}_2\Big\},
\end{align}
the minimization of $f(x,z)$ under the $\ell_0$-constraints 
\eqref{eq:ls_2l0} and the unconstrained problem \eqref{eq:spa.rob.reg_pen} are equivalent in that any global optimal solution to \eqref{eq:spa.rob.reg_pen} is globally optimal to the constrained problem 
 (and vice versa) 
 {if $S(\lambda_1,\lambda_2)\subset S(\underline{\lambda}_1,\underline{\lambda}_2)$ holds}.
\end{theorem}
See Section \ref{sec:proof:thm:exact.pen_spa.rob.reg} for the proof. 


\subsection{Extension of GIST}
To approach \eqref{eq:spa.rob.reg_pen}, we consider to solve a structured optimization problem defined as
\begin{equation}
\underset{x\in\mathbb{R}^p,z\in\mathbb{R}^N}{\textrm{minimize}}\quad F(x,z):=f(x,z)+g(x)+h(z).
\label{minPALM}
\end{equation}
Here we assume that $f:\mathbb{R}^p\times\mathbb{R}^N\to\mathbb{R}$ is 
$M$-smooth
in the sense of Assumption \ref{assump:L-smooth_xz}, and that $g:\mathbb{R}^p\to\mathbb{R}\cup\{\infty\}$ and $h:\mathbb{R}^N\to\mathbb{R}\cup\{\infty\}$ are both proper lsc. 

To solve \eqref{minPALM} we extend the alternating optimization algorithm known as Proximal Alternating Linearized Minimization (PALM) developed by Bolte et al.~\cite{bolte2014proximal}, by employing the BB rule and the nonmonotone line search as GIST does. 
The algorithm is described in Algorithm \ref{alg:GPALM}, which we call General PALM or \emph{GPALM} for short. 
In addition to the alternating proximal mappings defined by \eqref{eq:prox_wrt_x} and \eqref{eq:prox_wrt_z}, the nonmonotone line search is applied. 
We should note that \eqref{GISTPALM} holds if 
$\eta_t^x, \eta_t^z\geq M$, 
and hence we can always find $\eta_t^x$ and $\eta_t^z$ satisfying \eqref{GISTPALM}.
\begin{algorithm}[tb]
   \caption{GPALM}
   \label{alg:GPALM}
\begin{algorithmic}
   \STATE {\bf Input:} $r\geq1$, $\rho_1$, $\rho_2$ $>1$, $\sigma_1,\sigma_2\in(0,1)$, $0<\underline\eta\leq\overline\eta$
   \STATE {\bf Initialize:} $\eta_{t}^x=1$, $\eta_{t}^z=1$, $t=0$
   \REPEAT 
   \REPEAT 
 \STATE 
 Compute
\begin{align}
&x_{t+1}\in\mathrm{prox}_{g/\eta_t^x}\left(x_t-\frac{1}{\eta_t^x}\nabla_{x} f(x_t,z_t)\right),\label{eq:prox_wrt_x}\\
&z_{t+1}\in\mathrm{prox}_{h/\eta_t^z}\Big(z_t-\frac{1}{\eta_t^z}\nabla_{z} f(x_{t+1}, z_t)\Big),\label{eq:prox_wrt_z}\\
   &
   \eta_t^x\leftarrow \rho_1\eta_t^x, \notag\\
   &
   \eta_t^z\leftarrow \rho_2\eta_t^z.\notag
\end{align}

   \UNTIL{the following condition is satisfied
      \begin{align}
      F(x_{t+1},z_{t+1}) &\leq
\max\{F(x_{t-r+1},z_{t-r+1}),...,F(x_t,z_t)\}\notag\\
      &\quad - \frac{\sigma_1}{2}\eta_t^x\|x_{t+1} - x_t\|^2 
      - \frac{\sigma_2}{2}\eta_t^z\|z_{t+1}-z_t\|^2.
    \label{GISTPALM}
      \end{align}
   }
   \STATE Compute the BB step size for the next outer loop:
     \begin{align*}
     \eta^x &= \frac{\langle x_{t+1}-x_t,\nabla_xf(x_{t+1},z_{t+1})-\nabla_xf(x_t,z_t)\rangle}{\|x_{t+1}-x_t\|^2},\\
     \eta^z &= \frac{\langle z_{t+1}-z_t,\nabla_zf(x_{t+1},z_{t+1})-\nabla_zf(x_{t+1},z_t)\rangle}{\|z_{t+1}-z_t\|^2},\\
     &\eta_{t+1}^x=\min\{\overline\eta,\max\{\underline\eta,\eta^x\}\}, \\
     &\eta_{t+1}^z=\min\{\overline\eta,\max\{\underline\eta,\eta^z\}\},
     \end{align*}
     and set $t\leftarrow t+1$ 

   \UNTIL{
some termination condition holds.}
\end{algorithmic}
\end{algorithm}
In order to show the convergence of GPALM to a d-stationary point of \eqref{minPALM}, 
we make the following assumptions for problem \eqref{minPALM}.
\begin{assumption}
\label{assump:d-stationary_PALM}~
\begin{enumerate}
\item[(i)]
 $f$ is $M$-smooth in the sense of Assumption \ref{assump:L-smooth_xz}.  $g$ and $h$ are both proper lsc and directionally differentiable.
\item[(ii)]
$F$ is bounded below.
\item[(iii)]
$g$ and $h$ are prox-bounded, i.e., $g+\frac{\eta}{2}\|\cdot\|^2$ and $h+\frac{\eta}{2}\|\cdot\|^2$ are lower bounded for some $\eta>0$. 
\end{enumerate}
\end{assumption}\noindent
We give the following lemma. This lemma is an extended version of Lemma \ref{lemma:for_d-stat}.
\begin{lemma}\label{lemma:for_d-stat_GPALM}
Let $x_{t+1}$ and $z_{t+1}$ be given by \eqref{eq:prox_wrt_x} and \eqref{eq:prox_wrt_z}, respectively, and
\begin{align*}
\epsilon^x_t&:=\eta_t^x(x_{t+1} - x_t ) +\nabla_x f(x_t,z_t) -\nabla_x f(x_{t+1},z_{t+1}),\\
\epsilon^z_t&:=\eta_t^z(z_{t+1} - z_t ) +\nabla_z f(x_{t+1},z_t) -\nabla_z f(x_{t+1},z_{t+1}).
\end{align*}
Suppose that Assumption \ref{assump:d-stationary_PALM}(i) is satisfied. 
Then we have
\begin{align*}
-\|\epsilon^x_t\|&\|d_x\|-\|\epsilon^z_t\|\|d_z\|\leq
g'(x_{t+1};d_x) +h'(z_{t+1};d_z) \\
&+\langle\nabla_{x} f(x_{t+1},z_{t+1}),d_x\rangle + \langle\nabla_{z} f(x_{t+1},z_{t+1}),d_z\rangle.
\end{align*}
\end{lemma}
See Section \ref{sec:proof:lem:d-stat_palm} for the proof. 

We next give a lemma which ensures the counterpart of the condition \eqref{eq:lim_pgm} of PGM.
\begin{lemma}\label{lemma:lim_GPALM}
Suppose 
%
that Assumption \ref{assump:d-stationary_PALM} is satisfied. 
Let the sequence $\{(x_t,z_t):t\geq 0\}$ be generated by Algorithm \ref{alg:GPALM} (when the termination condition is ignored). 
If the generated sequence is bounded and $F$ is continuous on a compact set containing the sequence, then we have 
\begin{equation}\label{eq:lim_GPALM}
\lim_{t\to\infty} \|x_{t+1}-x_t\|=0,
\quad\text{and}\quad
\lim_{t\to\infty} \|z_{t+1}-z_t\| = 0.
\end{equation}
\end{lemma}
We can prove the lemma in a similar manner to the proof of Lemma 4 in \cite{wright2009sparse}. 
See Section \ref{sec:proof:lim_palm} for the proof. 

By using Lemmas \ref{lemma:for_d-stat_GPALM} and \ref{lemma:lim_GPALM}, we can prove the following theorem in a similar manner to the proof of  Theorem \ref{thm:d-stat_PGM}.
\begin{theorem}
\label{thm:d-stat_GPALM}
Suppose 
that Assumption \ref{assump:d-stationary_PALM} is satisfied. 
Let the sequence $\{(x_t,z_t,\eta^x_t,\eta^z_t):t\geq 0\}$ be generated by Algorithm \ref{alg:GPALM} (when the termination condition is ignored). 
If 
the generated sequence $\{(x_t,z_t)\}$ is bounded and $F$ is continuous on a compact set containing the sequence, then any accumulation point of $\{(x_t,z_t)\}$
is a d-stationary point.
\end{theorem}
\begin{proof}
Since the directional derivative of $F$ with respect to $(d_x,d_z)$ 
is defined as
\[
F'(x,z;d_x,d_z)=\lim_{\tau\to+0}\frac{F(x+\tau d_x,z+\tau d_z)-F(x,z)}{\tau},
\]
and we have 
\begin{equation*} 
F'(x,z;d_x,d_z)=g'(x;d_x)+\langle\nabla_{x}f(x,z),d_x\rangle
 +h'(z;d_z)+\langle\nabla_{z}f(x,z),d_z\rangle.
\end{equation*}

Let $(x^*,z^*)$ be any accumulation point of $\{(x_t,z_t)\}$ and let $\{(x_{t_i},z_{t_i})\}$ be a convergent subsequence with $(x_{t_i},z_{t_i})\to (x^*,z^*)$. By passing to further subsequences if necessary, we also assume that $\eta_{t_i}^x\to\eta_*^x$ and $\eta_{t_i}^z\to\eta_*^z$ for some $\eta_*^x, \eta_*^z\in(\underline{\eta},\overline{\eta})$. 
In a similar way to the proof of Theorem \ref{thm:d-stat_PGM},
we have
\[x^* \in {\rm Prox}_{g/\eta^x_*}\Big(x^* -\frac{1}{\eta^x_*}\nabla f(x^*,z^*)\Big),~
z^* \in {\rm Prox}_{h/\eta^z_*}\Big(z^* -\frac{1}{\eta^z_*}\nabla f(x^*,z^*)\Big).\]
Therefore, it follows from Lemma \ref{lemma:for_d-stat_GPALM} with $(x_{t+1},z_{t+1})=(x_t,z_t)=(x^*,z^*)$ and $(\eta_t^x,\eta_t^z)=(\eta^x_*,\eta^z_*)$ that
\begin{align*}
0 &\leq g'(x^*;d_x)+\langle\nabla_{x}f(x^*,z^*),d_x\rangle,\\
0 &\leq h'(z^*;d_z)+\langle\nabla_{z}f(x^*,z^*),d_z\rangle,
\end{align*}
which implies that 
for an accumulation point $(x^\ast,z^\ast)$, we have 
$
F'(x^*,z^*;d_x,d_z) \geq 0,
$ which is the desired result. 
\end{proof}
\begin{remark}
It is known that (plain) PALM 
clusters at an l-stationary point. 
Similar to GPALM, we can prove that PALM 
clusters at a d-stationary point. 
\end{remark}

Similar to PGM, the stopping criterion of Algorithm \ref{alg:GPALM} can be, for example, 
$\|x_{t+1}-x_t\|+\|z_{t+1}-z_t\|<\varepsilon$.

\section{Numerical comparisons}\label{sec:numerical}\label{sec:numerics}
This section reports numerical comparisons among PGMs and PDCAs.
We test  several algorithms in the two categories for  sparse regression problems in Section \ref{sec:SpaReg} and sparse robust regression in  Section \ref{sec:SpaRubReg}. 
All the algorithms were implemented with MATLAB 2017b and run on a laptop PC with 2.3 GHz Intel Core i5, 8 GB RAM, and macOS High Sierra. 

\subsection{Sparse regression}\label{sec:SpaReg}
In this subsection, we compare numerical performances of PGM (Algorithm \ref{alg:PGM}), GIST (Algorithm \ref{alg:GIST}), 
PDCAe (Algorithm \ref{alg:PDCAe}), and NEPDCA (Algorithm \ref{alg:NEPDCA}) for a few sparse regression problems. 
In Section \ref{sec:SpaReg_syn}, we use synthetic data sets, and 
in Section \ref{sec:SpaReg_real}, we use real data sets  from UCI Machine Learning Repository (\url{www.csie.ntu.edu.tw/~cjlin/libsvmtools/datasets/}).
\subsubsection{Case: Synthetic data sets}\label{sec:SpaReg_syn}
We first solved synthetic instances of the problem \eqref{eq:penalty_version} with $f(x)=\frac{1}{2}\|Ax-b\|^2$, 
where $A\in\mathbb{R}^{N\times p}$ and $b\in\mathbb{R}^{N}$. 
 Following 
Lu and Zhou~\cite{lu2019nonmonotone}, 
$A$ and 
$b$ were generated so that each instance would have a critical point $\tilde{x}$ which was not d-stationary (see \cite{lu2019nonmonotone} for the details). 
%

Table \ref{tb:SpaReg_LS} reports the number of iterations (Iter), CPU time in seconds (Time), the logarithm of the objective function value ($\ln F(x)$) and the number of nonzero components of the obtained solution ({$\|x\|_0$}), each showing the mean value over the randomly generated 30 instances. The stopping criterion of all the algorithm was $\|x_{t+1}-x_{t}\|\leq 10^{-8}$. 
For $l=1,2,...,10$, we chose the size and penalty parameters as $(p,N,K,\lambda)=(1000l,1000l,300l,10l)$, and generated 30 random instances for each $l$. Each run started from a common initial point which was randomly given as $x_0=\tilde{x}+0.01\nu$, where $\nu\in\mathbb{R}^p$ were drawn from $\mathrm{U}[-1,1]^p$, i.e., the uniform distribution over $[-1,1]^p$. 
In the table, the best values for each $l$ are given in boldface. 
%
 For GIST, we set $\sigma=10^{-3}$, ${\eta_0 = 1}$, $\underline\eta=10^{-8}$, $\overline\eta=10^8$, $r=4$, and $\rho=2$. 
 For PGM, we set $\eta_t=1.1L$ with $L=\lambda_{\max}(A^\top A)$, the largest eigenvalue of $A^\top A$. 
For PDCAe and NEPDCA, the parameters are set to the same as those of Wen et al.~\cite{wen2018proximal} and Liu et al.~\cite{lu2019enhanced}, respectively. 

We see from Table \ref{tb:SpaReg_LS} that while NEPDCA attained the smallest number of iterations, GIST  is the fastest in computation time. 
While GIST, PGM, 
and NEPDCA attained almost the same objective values, only PDCAe resulted in a significantly higher values, which is along with the observation of \cite{lu2019enhanced}. This seems to reflect the superiority of the algorithms guaranteed to converge to d-stationary points to that is only shown to converge to a critical point. 
Besides, PDCAe sometimes failed to satisfy the $\ell_0$-constraint, while the other 
four succeeded always.

\begin{table*}[t]
\begin{center}
\caption{Comparison for Sparse Least Square Regression}\label{tb:SpaReg_LS}
\scalebox{0.8}{
\begin{tabular}{|c||c|c|c|c||c|c|c|c|}
\hline
&  \multicolumn{4}{|c||}{Iter} &  \multicolumn{4}{|c|}{Time [sec.]} \\\hline
$l$ & GIST & PGM & PDCAe & NEPDCA & GIST & PGM& PDCAe & NEPDCA \\ \hline\hline
1 & 29.9  & 125.7   & 103.1  & \bf22.9  & \bf0.0266  & 0.0971   & 0.0763  & 0.0313    \\\hline
2 & 28.6  & 121.8   & 102.5  & \bf24.6  & \bf0.0829  & 0.3380    & 0.2714  & 0.1121   \\\hline
3 & 31.2  & 121.6    & 100.3  & \bf24.5  & \bf0.2250  & 0.8402   & 0.6878  & 0.2917   \\\hline
4 & 33.6  & 119.4    & 97.2  & \bf24.0  & \bf0.5891  & 1.9914   & 1.5999  & 0.6921   \\\hline
5 & 32.3  & 121.0    & 92.8  & \bf23.1  & \bf0.8018  & 2.8489    & 2.2544  & 1.0794   \\\hline
6 & 30.5  & 121.0    & 101.3  &\bf 23.9  & \bf1.0545  & 4.0385   & 3.3523  & 1.6297    \\\hline
7 & 33.5  & 119.7    & 94.0  & \bf23.8  & \bf1.5833  & 5.3777    & 4.3343  & 2.4651   \\\hline
8 & 33.6  & 122.8    & 100.4  & \bf23.4  &\bf 1.9788  & 6.9424   & 5.7080  & 3.0312   \\\hline
9 & 30.8  & 119.9    & 90.8  &\bf 24.1  & \bf2.2902  & 8.4873   & 6.4644  & 4.1936  \\\hline
10 & 31.8  & 120.5    & 95.6  & \bf24.5  & \bf2.8737  & 10.4521 & 8.3000  & 5.2553  \\ \hline
\hline
& \multicolumn{4}{|c||}{ln$(F(x))$} &  \multicolumn{4}{|c|}{{$\|x\|_0$}}\\\hline
$l$ & GIST & PGM & PDCAe & NEPDCA & GIST & PGM & PDCAe & NEPDCA\\ \hline\hline
1 & 0.22193  & 0.22241    & 1.40528  & \bf 0.22153  &\bf 300.0  &\bf 300.0  & 300.7  &\bf 300.0  \\\hline
2 & \bf0.55461  & 0.55689   & 1.40368  & 0.55734  & \bf600.0  & \bf600.0   & 600.4  & \bf600.0  \\\hline
3 & \bf 0.69628  & 0.70004   & 1.42486  & 0.69873  & \bf900.0  & \bf900.0   & 900.2  &\bf 900.0  \\\hline
4 & \bf0.81495  & 0.81792 & 1.34608  & 0.81949  &\bf 1200.0  &\bf 1200.0 & 1200.1  &\bf 1200.0  \\\hline
5 & \bf 0.90785  & 0.90920   & 1.52883  & 0.90842  & \bf1500.0  & \bf1500.0  & 1500.2  & \bf1500.0  \\\hline
6& 0.98646  & 0.98767  & 1.33702  & \bf 0.98598  &\bf 1800.0  & \bf1800.0    & 1800.1  & \bf1800.0  \\\hline
7 & \bf1.05668  & 1.05687    & 1.40063  & 1.05700  & \bf2100.0   & \bf2100.0  & 2100.1  &\bf 2100.0  \\\hline
8 & 1.10416  & 1.10560    & 1.10420  & \bf 1.10323  &\bf 2400.0   &\bf 2400.0  & \bf2400.0  & \bf2400.0  \\\hline
9 & 1.16587  & 1.16550    & 1.77804  &\bf  1.16546  & \bf2700.0    & \bf2700.0  & 2700.2  &\bf 2700.0  \\\hline
10 & 1.20919  & 1.20897   & 1.43860  & \bf 1.20768  &\bf 3000.0   &\bf 3000.0  &\bf 3000.0  & \bf3000.0\\ \hline
\end{tabular}
}
\end{center}
\end{table*}

\subsubsection{Case: Real data sets}\label{sec:SpaReg_real}
We solved real data sets of the problem \eqref{eq:penalty_version} with $f(x)=\frac{1}{2}\|Ax-b\|^2$, 
where $A\in\mathbb{R}^{N\times p}$ and $b\in\mathbb{R}^{N}$. We used the {\tt triazines} data set, which is of size $(N,p-1)=(186,60)$. To virtually incorporate the intercept in the model, all-one vector was added to the input data, 
and 
the variable $x_1$, which corresponds to the intercept, was excluded from the penalty function, i.e., we used $T_K(x_2,\cdots,x_p)$ instead of $T_K(x)$. 
Since PGM was apparently inferior to GIST 
in the previous experiment, we here compare only GIST, 
PDCAe, and NEPDCA. 

Figure \ref{fig:triazines} and Table \ref{tbl:triazines} show the results of the four algorithms. 
Each run 
started from the same initial point $x_0=0.1\nu$
, where $\nu\in\mathbb{R}^p$ was drawn 
from 
$\mathrm{U}[-1,1]^p$. 
The stopping criterion was $\|x_{t+1}-x_{t}\|\leq 10^{-6}$ for all the 
four algorithms. In Table \ref{tbl:triazines}, `90+' means that the algorithm 
did not fulfill the criterion within 90 [sec.]. 

The two DCA approaches were competitive only for small $\lambda$'s (i.e., $\lambda=0.001,0.1$), with which the obtained solutions failed to satisfy the $\ell_0$-constraint. On the other hand, for large $\lambda$'s (i.e., $\lambda=10,1000$), which are of interest in the exact penalty context, they 
were trapped at apparently worse solutions, 
resulting in higher objective values by overshooting the $\ell_0$-constraint. 
 %
\begin{table*}[h]
\begin{center}
\caption{GIST vs. PDCAe vs. NEPDCA in Sparse Least Square Regression with the data set {\tt triazines}}
    \label{tbl:triazines}
\scalebox{0.8}{

\begin{tabular}{|c|c||c|c|c||c|c|c|c||c|c|c||c|c|c|c|}
\hline
&&
 \multicolumn{3}{|c||}{Iter} &  \multicolumn{3}{|c|}{Time [sec.]}   \\ \hline
$\lambda$ & $K$ & GIST & PDCAe & NEPDCA & GIST & PDCAe & NEPDCA\\ \hline
0.001 & 9 & 2124  & 2500 & 1096 & 0.22487  & 0.19205 &0.16583  \\\hline
0.1 & 9 & 2817& 1898 & 2264 & 0.26507  & 0.14024 &0.30017 \\\hline
10 & 9 & 639 & 484 & 158 & 0.07050  & 0.04343 & 14.72229\\\hline
1000 & 9 & 723  & 330 & 5 & 0.07644  &0.03321 & 90+ \\\hline\hline
&&
  \multicolumn{3}{|c||}{$F(x)$} &  \multicolumn{3}{|c|}{$f(x)$}   \\ \hline
$\lambda$ & $K$ & GIST & PDCAe & NEPDCA & GIST& PDCAe & NEPDCA \\ \hline
0.001 & 9 & \bf1.32932 & 1.32934 & 1.32935 & \bf1.32713  & 1.32715 & 1.32715  \\\hline
0.1 & 9 &\bf1.41426 & 1.41428 &\bf1.41426 & \bf1.36434  & 1.36443 & \bf 1.36434 \\\hline
10 & 9 &\bf1.61452 & 1.95086 & 1.75285 & \bf1.61452 & 1.95086 & 1.75285 \\\hline
1000 & 9 &\bf 1.58680 & 2.17224 & 3.21618 & \bf1.58680 & 2.17224 & 3.21618 \\\hline\hline
&&
  \multicolumn{3}{|c||}{$\|(x_2,\cdots,x_p)\|_0$}  \\ \cline{1-5}
$\lambda$ & $K$ & GIST & PDCAe & NEPDCA   \\\cline{1-5}
0.001 & 9 & 59 & 58 & 59 \\\cline{1-5}
0.1 & 9 &32  & 32 & 32 \\\cline{1-5}
10 & 9 &\bf 9 & 5 & 6 \\\cline{1-5}
1000 & 9 &\bf 9& 3 & 3\\\cline{1-5}
\end{tabular}
}
\end{center}
\end{table*}

\begin{figure}[h]
  \begin{center} 
     \includegraphics[clip,width=\linewidth]{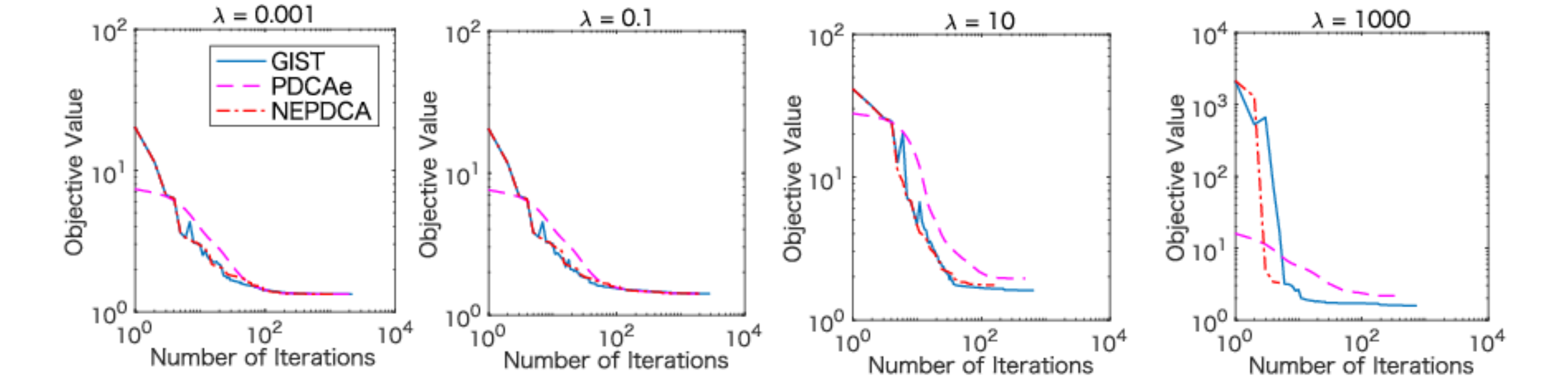}
    \caption{Objective values vs. iterations loglogplot (data set: {\tt triazines})}
    \label{fig:triazines}
  \end{center}
\end{figure}

We next 
applied GIST 
and NEPDCA 
to  the logistic regression problem \eqref{eq:penalty_version} with $f(x)=\frac{1}{N}\sum_{i=1}^N{\rm ln}(1+\exp(-b_i(x_1+\sum_{j=2}^p a_{ij} x_j))$ and $g=T_K(x_2,\cdots,x_p)$.  
Since Lu and Zhou \cite{lu2019nonmonotone} reported that PDCAe was inferior to NEPDCA and the previous experiment confirmed a similar result, we here compare only GIST, 
and NEPDCA as the three best methods.
We used another relatively big data set {\tt real-sim}, 
 which is of size $(N,p-1)=(72309,~20958)$. 
In each run, 
the same initial point $x_0=0.1\nu$ was 
used where $\nu\in\mathbb{R}^p$ were 
drawn 
from 
$\mathrm{U}[-1,1]^p$. 
The upper time limit was set to 600 [sec.].
The stopping criterion was $\|x_{t+1}-x_{t}\|\leq 10^{-6}$ for all the three algorithms.
No algorithms 
fulfilled the criterion within 
the time limit. 

 From Table \ref{tb:realsim}, we see that 
while GIST 
resulted in the designated cardinality, 
NEPDCA resulted in overly sparse solutions. 
The smaller number of iterations of NEPDCA reflects the situation where lots of zero components were introduced by the soft-thresholding and the active set size $|\mathcal{A}_{\delta}(x_t)|$ grew exponentially at each iteration. 

 From all the observations above, 
we see that the PGM-based methods 
tend to find 
better solutions in a more efficient way than the PDCA-based methods for the sparse regression problems. 
In particular, GIST was the best through our experiments. 

\begin{table*}[h]
\begin{center}
\caption{GIST vs. NEPDCA for Sparse Logistic Regression with the data set {\tt real-sim}}\label{tb:realsim}
\scalebox{0.72}{

\begin{tabular}{|c|c||c|c||c|c||c|c||c|c||}
\hline
&& \multicolumn{2}{|c||}{Iter}  &  \multicolumn{2}{|c|}{$\|(x_2,\cdots,x_p)\|_0$} &   \multicolumn{2}{|c||}{$F(x)$} &  \multicolumn{2}{|c|}{$f(x)$}   \\ \hline
 $\lambda$ & $K$ & GIST &NEPDCA & GIST &NEPDCA  & GIST &NEPDCA & GIST &NEPDCA  \\ \hline
0.001 & 19& 855  & 721 &\bf19 &\bf19 &  \bf0.31253  & 0.32227 & \bf0.31253  & 0.32227  \\\hline
0.1  & 19&13  & 12  & \bf19   & 18 & \bf0.31560  &0.61738& \bf0.31560 & 0.61738   \\\hline
10  & 19&13  & 4 &\bf 19   & 18 & \bf0.31648 & 0.61496& \bf0.31648  & 0.61496  \\\hline
1000  &19& 13  & 4  & \bf19    & 18 & \bf0.31648  & 0.61496 &\bf0.31648  & 0.61496\\\hline
\end{tabular}
}
\end{center}
\end{table*}

\subsection{Sparse robust regression}\label{sec:SpaRubReg}
Lastly we compared three methods for the sparse least trimmed square regression problem \eqref{eq:spa.rob.obj}--\eqref{eq:ls_2l0}.  
We applied GPALM (Algorithm \ref{alg:GPALM}) and (plain) PALM \cite{bolte2014proximal} to the 
problem of the form \eqref{eq:spa.rob.reg_pen} with $f(x,z)=\frac{1}{2}\|Ax-b-z\|^2$ where $A\in\mathbb{R}^{N\times p}$ and $b\in\mathbb{R}^{N}$, and applied the PDCAe-projection approach (PDCAe-proj, for short) of Liu et al.~\cite{liu2019Arefined} to \eqref{eq:LiuRSR}. 

 For the four common scales $l=1,...,4$, the size and penalty parameters were given as $(p,N,K,\kappa,\lambda_1,\lambda_2)=(2560l,720l,80l,15l,\tilde\lambda l,\tilde\lambda l)$, where the three different levels of penalty $\tilde\lambda=0.01,1,100$ were considered; 
we generated 30 random instances for each combination of the parameters. 
Each run of every algorithm started from the same initial point $x_0=0.01\nu_x$ and $z_0=0.01\nu_z$, where $\nu_x\in\mathbb{R}^p$ and $\nu_z\in\mathbb{R}^N$ were drawn 
from $U[-1,1]^{p+N}$; 
The stopping criterion was $\|x_{t+1}-x_{t}\|+\|z_{t+1}-z_{t}\|\leq 10^{-6}$ for all the three algorithms. 
For GPALM, we set $\sigma_1=\sigma_2=10^{-3}$, {$\eta_0^x=1$, $\eta_0^z=1$,} $\underline\eta=10^{-8}$, $\overline\eta=10^8$, $r=6$, and $\rho_1=\rho_2=2$. 
For PALM, we set $\eta_t^x=1.1\lambda_{\max}(A^\top A)$, and $\eta_t^z=1.1$. 
For PDCAe-proj, 
the associated parameters are set to the same values as 
in \cite{liu2019Arefined}.

For small penalty parameters, we see from Table \ref{tb:SpaRobReg_LS} that PCDAe-proj 
was better than GPALM and PALM in computation time and the number of iterations, but we need to be careful in looking at the quality of the obtained solutions. While PDCAe-proj attained the designated cardinality of $z$, which corresponds to the number of outlying samples, it was forced to be fulfilled at each iteration due to the projection operation of $z$ to the $\ell_0$-constraint. On the other hand, PDCAe-proj resulted in less sparsity for $x$, which corresponds to selected variables. From the columns of `$\ln (F(x,z))$' we see that while GPALM, PALM, and PDCAe-proj attained almost the same objective values for small $\lambda$'s, GPALM attained slightly but consistently better average values. 

For large penalty parameters, GPALM performed better than the other two in solution quality and computation time. Specifically, it seems that PDCAe-proj resulted in a bad critical point while the two PGM-based methods, GPALM and PALM, constantly attained the designated cardinality (with equality) for $x$ and smaller objective values. Furthermore, GPALM performed better than PALM on average, which indicates that the BB-rule and the non-monotone line search improve on not only the computation efficiency but also the solution quality.


%

\begin{table*}[h]
\begin{center}
\caption{GPALM vs. PALM vs. PDCAe-proj for Sparse Robust Regression for the synthetic data}\label{tb:SpaRobReg_LS}
\scalebox{0.7}{
\begin{tabular}{|c|c||c|c|c||c|c|c||c|c|c||c|c|c|}
\hline
& & \multicolumn{3}{|c||}{Iter} & \multicolumn{3}{|c||}{Time [sec.]} & \multicolumn{3}{|c|}{$\|x\|_0$}   \\\hline
$\lambda$ & $n$ & GPALM & PALM & PDCAe & GPALM & PALM & PDCAe & GPALM & PALM & PDCAe \\
\hline\hline
0.01 & 2560 & {809.1}  & 1147.4  & \bf252.8  & 1.491  & 1.262  & \bf{0.255} & 268 & 269.1 & 322.2 \\\hline
0.02 & 5120 & {409.8}  & 586.3  &\bf 206.5  & 2.903  & 2.257  & \bf{0.737} & 245.3 & 248.9 & 255.8 \\\hline
0.03 & 7680 & 277.1  & 423.7  & \bf{202.0}  & 5.957  & 4.787  & \bf{2.167}  & 243.7 & 250.4 & 245 \\\hline
0.04 & 10240 & 215.3  & 370.1  & \bf{201.6}  & 8.099  & 7.134  & \bf{3.724}   & \bf{320} & 320.8 & {313.4}  \\\hline
1 & 2560 & \bf{33.8}  & 285.9  & 141.9  & \bf{0.077}  & 0.315  & 0.142  & \bf{80} & \bf{80} & {56.2}  \\\hline
2 & 5120 & \bf{37.2}  & 309.0  & 115.7  & \bf{0.285}  & 1.170  & 0.411  & \bf{160} & \bf{160} & {73.2}\\\hline
3 & 7680 & \bf{37.2}  & 298.2  & 108.5  & \bf{0.843}  & 3.350  & 1.167  & \bf{240} & \bf{240} & {109.9} \\\hline
4 & 10240 & \bf{36.1}  & 313.1  & 108.2  & \bf{1.442}  & 6.153  & 2.041   & \bf{320} & \bf{320} & {148.6}  \\\hline
100 & 2560 & \bf{35.9}  & 297.4  & 103.3  & \bf{0.069}  & 0.313  & 0.094  & \bf{80} & \bf{80} & {37.5}  \\\hline
200 & 5120 & \bf{37.2}  & 296.8  & 107.6  & \bf{0.264}  & 1.085  & 0.358  & \bf{160} & \bf{160} & {74.6} \\\hline
300 & 7680 & \bf{36.3}  & 310.6  & 107.6  & \bf{0.794}  & 3.402  & 1.116 & \bf{240} & \bf{240} & {111.5} \\\hline
400 & 10240 & \bf{35.9}  & 307.8  & 107.2  & \bf{1.424}  & 6.016  & 2.020  & \bf{320} & \bf{320} & {148.2} \\\hline
\hline
&  & \multicolumn{3}{|c||}{ln$(F(x,z))$} & \multicolumn{3}{|c||}{ln$(f(x,z))$}  & \multicolumn{3}{|c|}{$\|z\|_0$} \\\hline
$\lambda$ & $n$ & GPALM & PALM & PDCAe & GPALM & PALM & PDCAe & GPALM & PALM & PDCAe \\
\hline\hline
0.01 & 2560 &  \bf{-1.6264}  & -1.6243  & -1.6098  & \bf{-1.9440}  & -1.9432  & -1.8820  & 95.9 & 96.2 & \bf{15} \\\hline
0.02 & 5120  & \bf{-1.2337}  & -1.2275  & -1.2299  & \bf{-1.2935}  & -1.2923  & -1.2789   & 60.5 & 60.8 & \bf{30} \\\hline
0.03 & 7680 & \bf-{1.0406}  & -1.0256  & -1.0379  & \bf{-1.0423}  & -1.0303  & -1.0393  & 47.6 & 47.5 & \bf{45} \\\hline
0.04 & 10240& \bf{-0.9177}  & -0.8947  & -0.9004  & \bf{-0.9177}  & -0.8950  & -0.9004   & \bf{60} & \bf{60} & \bf{60} \\\hline
1 & 2560 &  \bf{-1.5245}  & -1.4575  & 0.7687  & \bf{-1.5245}  & -1.4575  & 0.7687 & \bf{15} & \bf{15} & \bf{15} \\\hline
2 & 5120   & \bf{-1.2166}  & -1.1555  & 1.6980  & \bf{-1.2166}  & -1.1555  & 1.6980   & \bf{30} & \bf{30} & \bf{30} \\\hline
3 & 7680  & \bf{-1.0447}  & -0.9944  & 2.0364  & \bf{-1.0447}  & -0.9944  & 2.0364  & \bf{45} & \bf{45} & \bf{45} \\\hline
4 & 10240   & \bf{-0.9167}  & -0.8607  & 2.1624  & \bf{-0.9167}  & -0.8607  & 2.1624  & \bf{60} & \bf{60} & \bf{60} \\\hline
100 & 2560   & \bf{-1.5177}  & -1.3595  & 1.5554  & \bf{-1.5177}  & -1.3595  & 1.5554  &\bf{15} & \bf{15} & \bf{15} \\\hline
200 & 5120  & \bf{-1.2182}  & -1.1535  & 1.8679  & \bf{-1.2182}  & -1.1535  & 1.8679  & \bf{30} & \bf{30} & \bf{30} \\\hline
300 & 7680  & \bf{-1.0417}  & -0.9883  & 2.0425  & \bf{-1.0417}  & -0.9883  & 2.0425  &  \bf{45} & \bf{45} & \bf{45} \\\hline
400 & 10240  & \bf{-0.9228}  & -0.8698  & 2.1658  & \bf{-0.9228}  & -0.8698  & 2.1658  & \bf{60} & \bf{60} & \bf{60}\\\hline
\end{tabular}
}
\end{center}
\begin{flushleft}
\scriptsize In this table `PDCAe' stands for PDCAe-proj.
\end{flushleft}
\end{table*}



\section{Concluding remarks}\label{sec:conclusion}
In this paper we show that PGM converges to a d-stationary point without any special modification and so do its derivatives, such as GIST 
and GPALM. 
Numerical results 
demonstrate stably better performances of those PGM-based methods over PDCAe, which is only proved to converge to a critical point, in the context of the sparse optimization problems given by the exact penalty form. 
Among the PDCA-based methods, NEPDCA performed better than PDCAe, but it took longer than the PGM-based methods, especially when the penalty parameter is large because of the growing size of the active set at each iteration of NEPDCA. 
While both PDCAe and NEPDCA performed in a comparable manner to GIST or GPALM for small penalty parameters, 
the obtained solutions 
failed to fulfill the 
$\ell_0$-constraint then. This is because the soft-thresholding operation involved in the PDCAs results in an excessively sparse solution when the penalty parameter 
is large. 
Consequently, 
the PDCA-based methods are less attractive for the sparse optimization problem based on the DC decomposition \eqref{eq:l1-largeK}, 
and PGM-methods are more suitable. Especially, GIST seems to be the most promising among the candidates considered in this paper. 


A possible downside of 
the PGM-based methods can be found when the proximal operation of the nonsmooth nonconvex function $g$ is not computationally tractable. 
In such a case, PDCA could be advantageous. Especially 
NEPDCA (or its derivative) may be helpful compared to the other PDCA versions which are only proved to converge to critical points since NEPDCA is proven to converge to a d-stationary point. 



\paragraph{acknowledgement}
S. Nakayama is supported in part by JSPS KAKENHI Grant 20K11698 and 20K14986.
J. Gotoh is supported in part by JSPS KAKENHI Grant 19H02379, 19H00808, and 20H00285. 

\appendix
\section{Appendix}
\subsection{Proof of Lemma \ref{lemma:for_d-stat}}\label{sec:proof_lem:d-stat}
Let
\[
q(x;x_t) := \frac{\eta_t}{2}\left\|x-\left(x_t-\frac{1}{\eta_t}\nabla f(x_t)\right)\right\|^2.
\]
Since $x_{t+1}$ is a minimizer of $\min Q(x;x_t) := g(x) + q(x;x_t)$, we have 
$0 \in \hat\partial Q(x_{t+1};x_t),$ 
which yields
\begin{align*}
0 &\leq  \liminf_{y\neq x_{t+1}, y \to x_{t+1}}\frac{ Q(y;x_t ) - Q(x_{t+1};x_t )}{\|y-x_{t+1}\|}\\
 &\leq \lim_{\tau \to +0}\frac{ Q(x_{t+1}+\tau d;x_t ) - Q(x_{t+1};x_t )}{\tau\|d\|}\\
 &\leq \lim_{\tau \to +0}\frac{ g(x_{t+1}+\tau d) - g(x_{t+1})}{\tau\|d\|}  +\lim_{\tau \to +0}\frac{ q(x_{t+1}+\tau d;x_t ) - q(x_{t+1};x_t )}{\tau\|d\|}
 \end{align*}
for all $d \in \mathbb{R}^p$.
Therefore, we have
\[
0\leq g'(x_{t+1};d) + \langle\eta_t(x_{t+1} - x_t ) +\nabla f(x_t),d\rangle,
\]
which implies
\[
-
\langle\epsilon_t,d\rangle \leq g'(x_{t+1};d) + \langle\nabla f(x_{t+1}),d\rangle
 = F'(x_{t+1};d).
\]
Therefore, 
$
-\|\epsilon_t\|
\|d\|\leq 
F'(x_{t+1};d).$
\hfill$\Box$

\subsection{Proof of Theorem \ref{thm:exact.pen_spa.rob.reg}}\label{sec:proof:thm:exact.pen_spa.rob.reg}
%
We only show the condition for $\lambda_1$ since we can show that for $\lambda_2$ in the same manner. 
We prove the statement by contradiction. 
Suppose that $\|x^*\|_0>K$ and $x^*_{(K+1)}>0$.  
It follows from Assumption \ref{assump:L-smooth_xz} that for any $x_1,~x_2\in\mathbb{R}^p$, $\tilde{z}\in\mathbb{R}^N$
\begin{align*}
\|\nabla_x f(x_1,\tilde{z}) - \nabla_xf(x_2,\tilde{z})\|
\leq
\left\|
\begin{matrix}
\nabla_x f(x_1,\tilde{z}) -\nabla_x f(x_2,\tilde{z})\\
 \nabla_z f(x_1,\tilde{z})-\nabla_z f(x_2,\tilde{z})
 \end{matrix}\right\|
\leq M
\left\|
x_1-x_2
\right\|,
\end{align*}
which means
\[f(x_1,\tilde{z}) \leq f(x_2,\tilde{z}) + \nabla_{x}f(x_2,\tilde{z})^\top(x_1-x_2)+\frac{M}{2}\|x_2-x_1\|^2.\]
Let $\tilde{x}:=x^*-x_i^*e_i$, then the above inequality yields
\[f(\tilde{x},z^*) \leq f(x^*,z^*) {-} \nabla_{x}f(x^*,z^*)^\top(x_i^*e_i)+\frac{M}{2}(x_i^*)^2,\]
where $i=(K+1)$.
Therefore, we obtain
\begin{align*}
F(x^*,z^*)-F(\tilde{x},z^*)
 & =f(x^*,z^*)+\lambda_1T_{K}(x^*)+\lambda_2T_{\kappa}(z^*)\\
 & \qquad\quad-
 \left( f(\tilde{x},z^*)+\lambda_1T_{K}(\tilde{x})+\lambda_2T_{\kappa}(z^*)\right)\\
 & \geq\nabla_{x}f(x^*,z^*)^\top(x_i^*e_i)-\frac{M}{2}(x_i^*)^2+\lambda_1|x^*_i|\\
 & \geq |x^*_i|(\lambda_1-\|\nabla_{x}f(x^*,z^*)\|-\frac{MC_x}{2}).
\end{align*}
Noting that
\begin{align*}
\|\nabla_{x}f(x^*,z^*)\|
&\leq \|\nabla_{x}f(0,0)\|+\|\nabla_{x}f(x^*,z^*)-\nabla_{x}f(0,0)\|\\
&\leq \|\nabla_{x}f(0,0)\|+M(\|x^*\|+\|z^*\|)\\
&\leq \|\nabla_{x}f(0,0)\|+M(C_x+C_z),
\end{align*}
and \eqref{eq:cond_lambda1}, we have
\begin{equation*}
F(x^*,z^*)-F(\tilde{x},\tilde{z})
\geq |x^*_i|[\lambda_1-\|\nabla_{x}f(0,0)\|-M(\frac{3}{2}C_x+C_z)]>0,
\end{equation*}
which contradicts the optimality of $x^*$. 
Similarly, we can derive the condition for $\lambda_2$. 
\hfill$\Box$

\subsection{Proof of Lemma \ref{lemma:for_d-stat_GPALM}}\label{sec:proof:lem:d-stat_palm}
Let 
\begin{align*}
Q_x(x|x_t,z_t)&:=g(x)+q_x(x|x_t,z_t),\\
Q_z(z|x_{t+1},z_t)&:=h(z)+q_z(z|x_{t+1},z_t),
\end{align*}
with
\begin{align*}
q_x(x|x_t,z_t)&:=\frac{\eta^x_t}{2}\left\|x-\Big(x_t-\frac{1}{\eta_t^x}\nabla_{x}f(x_t,z_t)\Big)\right\|^2,\\
q_z(z|x_{t+1},z_t)&:=\frac{\eta^z_t}{2}\left\|z-\Big(z_t-\frac{1}{\eta_t^z}\nabla_{z}f(x_{t+1},z_t)\Big)\right\|^2.
\end{align*}
 From \eqref{eq:prox_wrt_x} and \eqref{eq:prox_wrt_z}, $x_{t+1}$ and $z_{t+1}$ minimize $q_x(x|x_t,z_t)$ and $q_z(z|x_{t+1},z_t)$, respectively, and accordingly we have 
\[
0\in \hat{\partial}_x Q_x(x_{t+1}|x_t,z_t), \quad  
0\in \hat{\partial}_z Q_z(z_{t+1}|x_{t+1},z_t).
\]

Similarly to the proof of Lemma \ref{lemma:for_d-stat}, we have
\begin{align*}
0&\leq g'(x_{t+1};d_x) + \langle\eta^x_t(x_{t+1} - x_t) +\nabla_x f(x_t,z_t),d_x\rangle,\\
0&\leq h'(z_{t+1};d_z) + \langle\eta^z_t(z_{t+1} - z_t) +\nabla_z f(x_{t+1},z_t),d_z\rangle,
\end{align*}
which implies
%
\begin{align*}
-\|\epsilon^x_t\|\|d_x\|\leq g'(x_{t+1};d_x) + \langle\nabla_{x} f(x_{t+1},z_{t+1}), d_x\rangle,\\
-\|\epsilon^z_t\|\|d_z\|\leq h'(z_{t+1};d_z) + \langle\nabla_{z} f(x_{t+1},z_{t+1}), d_z\rangle.
\end{align*}
\hfill $\Box$
\subsection{Proof of Lemma \ref{lemma:lim_GPALM}}
\label{sec:proof:lim_palm}
Denoting 
\begin{equation*}\label{eq:lemma_palm_nonmono1}
\phi(t) =  \underset{j=\max\{0,t-r+1\},...,t}{\rm argmax}F(x_j,z_j),
\end{equation*}
we can rewrite \eqref{GISTPALM} as
\begin{equation}\label{eq:lemma_palm_nonmono2}
F(x_{t+1},z_{t+1}) \leq F(x_{\phi(t)},z_{\phi(t)}) -\frac{\sigma_1}{2}\eta_t^x\|x_{t+1} - x_t\|^2 
   - \frac{\sigma_2}{2}\eta_t^z\|z_{t+1}-z_t\|^2.
\end{equation}
Then we have
\begin{align*}
F(x_{\phi(t+1)},z_{\phi(t+1)}) &= 
   \max_{j=0,1,...,\min\{r-1,t+1\}} F(x_{t+1-j},z_{t+1-j}) \\
   &= \max\left\{\max_{j=1,...,\min\{r-1,t+1\}} F(x_{t+1-j},z_{t+1-j}),F(x_{t+1},z_{t+1})\right\}\\
   &\leq\max\Big\{F(x_{\phi(t)},z_{\phi(t)}),F(x_{\phi(t)},z_{\phi(t)}) \\
    & \qquad \qquad-\frac{\sigma_1}{2}\eta_t^x\|x_{t+1} - x_t\|^2  - \frac{\sigma_2}{2}\eta_t^z\|z_{t+1}-z_t\|^2\Big\}\\
    &= F(x_{\phi(t)},z_{\phi(t)}),
\end{align*}
which implies that the sequence $\{ F(x_{\phi(t)},z_{\phi(t)}):t\geq 0\}$ is monotonically 
decreasing. 
Therefore,  since $F$ is bounded below, there exists $\bar{F}$ such that 
\begin{equation}\label{eq:palm_limit}
\lim_{t\to \infty}F(x_{\phi(t)},z_{\phi(t)})=\bar{F}.
\end{equation}
By applying \eqref{eq:lemma_palm_nonmono2} with $t$ replaced by $\phi(t)-1$, we obtain
\begin{align*}
F(x_{\phi(t)},z_{\phi(t)}) &\leq F(x_{\phi(\phi(t)-1)},z_{\phi(\phi(t)-1)}) -\frac{\sigma_1}{2}\eta_{\phi(t)-1}^x\|x_{\phi(t)}  - x_{\phi(t)-1}\|^2 \\
 &  \qquad\qquad
   - \frac{\sigma_2}{2}\eta_{\phi(t)-1}^z\|z_{\phi(t)}-z_{\phi(t)-1}\|^2.
\end{align*}
Therefore, it follows from \eqref{eq:palm_limit} that 
\[
\lim_{t\to\infty}  \sigma_1\eta_{\phi(t)-1}^x\|x_{\phi(t)}  - x_{\phi(t)-1}\|^2  + \sigma_2\eta_{\phi(t)-1}^z\|z_{\phi(t)}-z_{\phi(t)-1}\|^2= 0.
\]
Since  $\eta_{\phi(t)-1}^x$, $\eta_{\phi(t)-1}^z\geq\underline\eta$, we have 
\begin{equation}\label{eq:lemma_palm_lim1}
\lim_{t\to\infty}  \|x_{\phi(t)}  - x_{\phi(t)-1}\|^2 =0 \quad
\text{and}\quad
\lim_{t\to\infty}  \|z_{\phi(t)}-z_{\phi(t)-1}\|^2= 0.
\end{equation}
With \eqref{eq:palm_limit}, \eqref{eq:lemma_palm_lim1}, 
the boundedness of the sequence, and continuity of $F$, 
we have
\begin{align*}
\bar{F} & = \lim_{t\to \infty}F(x_{\phi(t)},z_{\phi(t)})\\
& = \lim_{t\to \infty}F(x_{\phi(t)-1} + (x_{\phi(t)}  - x_{\phi(t)-1}) ,z_{\phi(t)-1}+(z_{\phi(t)}-z_{\phi(t)-1}))\\
 &=\lim_{t\to \infty}F(x_{\phi(t)-1},z_{\phi(t)-1}).
 \end{align*}

Next we 
prove, by induction, that the following 
for all $j\geq1$,
\begin{align}
\lim_{t\to\infty}  \|x_{\phi(t)}  - x_{\phi(t)-j}\|^2 =0 \quad
\text{and}\quad
\lim_{t\to\infty}  \|z_{\phi(t)}-z_{\phi(t)-j}\|^2= 0,\label{eq:lemma_palm_lim2}\\
\lim_{t\to \infty}F(x_{\phi(t)-j},z_{\phi(t)-j})=\bar{F}.
\label{eq:palm_limit2}
\end{align}
We have already 
observed that the results hold for $j=1$. 
Suppose that \eqref{eq:lemma_palm_lim2} and \eqref{eq:palm_limit2} hold for $j$. 
From \eqref{eq:lemma_palm_nonmono2} with $t$ replaced by $\phi(t)-j-1$, we get
\begin{align*}
F(x_{\phi(t)-j},z_{\phi(t)-j}) &\leq F(x_{\phi(\phi(t)-j-1)},z_{\phi(\phi(t)-j-1)})\\
&\qquad -\frac{\sigma_1}{2}\eta_{\phi(t)-j-1}^x\|x_{\phi(t)-j} - x_{\phi(t)-j-1}\|^2  \\
& \qquad   - \frac{\sigma_2}{2}\eta_{\phi(t)-j-1}^z\|z_{\phi(t)-j}-z_{\phi(t)-j-1}\|^2,
\end{align*}
which ensures \eqref{eq:palm_limit2}. 
 Furthermore, it follows from $\eta_{l}^x$, $\eta_{l}^z\geq\underline\eta$ for all $l$ that \eqref{eq:lemma_palm_lim2}.
Hence, we have \eqref{eq:lim_GPALM}.
\hfill $\Box$


\end{document}